\documentclass[a4paper, 10pt]{article}
\usepackage[utf8]{inputenc}
\usepackage[left=20 mm, bottom=20 mm, right=20 mm, top=20 mm]{geometry}
\usepackage{amsmath, amssymb, amsthm, csquotes, tikz, float}
\usepackage{hyperref}

\usepackage[english]{babel}
\usepackage{graphicx,epstopdf,epsfig}
\usepackage{amsfonts,epsfig,fancyhdr, graphics, hyperref,amsmath, tikz,amssymb}
\usepackage{geometry}
\usepackage{tikz}
\usetikzlibrary{decorations.markings}
\newtheorem{theorem}{Theorem}[section]

\newtheorem{remark}{Remark}[section]

\newtheorem{prop}{Proposition}[section]
\newtheorem{defn}{Definition}[section]
\newtheorem{cor}{Corollary}[section]
\newtheorem{eg}{Example}[section]
\numberwithin{equation}{section}

\newcommand{\Names}{Partha Rana} 
\newcommand{\Title}{Some Symmetric Sign Patterns Requiring Full $P$-vertices}

\def\det{{\rm det}}
\def\sgn{{\rm sgn}}
\def\In{{\rm In}}

\def\min{{\rm min}}

\def\dist{{\rm dist}}
\def\rank{{\rm rank}}

\begin{document}
\title{\Title\thanks{
		Corresponding Author: Partha Rana.}
	\author{Partha\ Rana\thanks{Department of Mathematics, Indian Institute of Technology Guwahati, Guwahati,
			Assam 781039, India.\\
			E-mail addresses: r.partha@iitg.ac.in (P. Rana). }}
}

\markboth{\Names}{\Title}

\date{}

\maketitle

\begin{abstract}
A sign pattern is a matrix whose entries belong to $\{+, -, 0\}$. Let $\mathcal{P}$ be a symmetric sign pattern and $A$ a real symmetric matrix in its qualitative class. A vertex $v_j$ of the underlying graph of $\mathcal{P}$ is called a $P$-vertex if
$
m_{A(j)}(0)-m_A(0)=1,
$
where $A(j)$ is the principal submatrix obtained by deleting the $j$-th row and column of $A$, and $m_A(0), ~m_{A(j)}(0)$ denotes the algebraic multiplicity of the eigenvalue $0$ of $A, ~A(j)$, respectively. We say that $\mathcal{P}$ requires full $P$-vertices if every symmetric matrix in its qualitative class has all vertices as $P$-vertices.
In this paper, we investigate structural conditions under which symmetric sign patterns require full $P$-vertices. We establish necessary and sufficient conditions for several classes of sign patterns to require full $P$-vertices. In particular, we prove that a tree sign pattern with a $0$-diagonal requires full $P$-vertices if and only if its underlying graph admits a perfect matching. We also derive necessary and sufficient conditions for sign patterns whose underlying graphs contain cycles but no loops to require full $P$-vertices.
\end{abstract}

\noindent {\bf Key words:} Sign pattern, Tree, $P$-vertex, Perfect matching, Tridiagonal sign pattern.

\noindent {\bf AMS Subject Classification:} 05C50, 15A18, 15B35.
\maketitle



	\section{Introduction}
	A matrix whose entries belong to the set $\{+, -, 0\}$ is called a 
	\textit{sign pattern matrix} (or simply, a \textit{sign pattern} or a 
	\textit{pattern}). The set of all $n\times n$ sign patterns is denoted 
	by $\mathcal{Q}_n$. For $\mathcal{P} = [p_{ij}] \in \mathcal{Q}_n$, the set  
	$$
	\mathcal{Q}(\mathcal{P}) =
	\left\{
	A = [a_{ij}] \in \mathbb{R}^{n\times n} :
	\sgn(a_{ij}) = p_{ij} \ \text{for all} \ i,j = 1, 2, \dots, n
	\right\}
	$$
	is called the \textit{qualitative class} of $\mathcal{P}$. 	A sign pattern $\mathcal{P} = [p_{ij}] \in \mathcal{Q}_n$ is said to be \textit{symmetric} if both $p_{ij}$ and $p_{ji}$ have the same sign $\text{for all} \ i,j = 1, 2, \dots, n$.
	We define the class of all symmetric matrices in the qualitative class of the symmetric sign pattern $\mathcal{P}$ by
	$$\mathcal{Q}_{\mathrm{SYM}}(\mathcal{P})=\{A\in \mathcal{Q}(\mathcal{P}): A=A^{\mathsf T}\}.$$

A sign pattern $\mathcal{P} \in \mathcal{Q}_n$ is called \textit{sign nonsingular} if every matrix $A \in \mathcal{Q}(\mathcal{P})$ is nonsingular. That is, in the standard expansion of $\det(\mathcal{P})$ into $n!$ terms, there exists at least one nonzero term and all nonzero terms have the same sign (see, e.g., \cite{2001}). The sign pattern $\mathcal{P}$ is said to be \textit{sign singular} if every matrix $A \in \mathcal{Q}(\mathcal{P})$ is singular, or equivalently, if $\det(\mathcal{P}) = 0$.

	Two matrices $A_1$ and $A_2$ are called \textit{equivalent} if one can be transformed into the other by applying a sequence of operations consisting of permutation similarity, signature similarity, negation, and transposition. Such an equivalence relation preserves several matrix properties; in particular, zreo is an eigenvalue of $A_1$ with multiplicity $m$ if and only if zero is an eigenvalue of $A_2$ with multiplicity $m$.

	The \textit{signed directed graph} $D(\mathcal{P}) = (V, E)$ of $\mathcal{P} = [p_{ij}] \in \mathcal{Q}_n$ is the directed graph with vertex set $V = \{v_1, v_2, \dots, v_n\}$ and directed edge $(v_i, v_j) \in E$ if and only if $p_{ij} \neq 0$. The directed edge $(v_i,v_j)$ is associated with a $+$ if $p_{ij}=+$ and associated with a $-$ if $p_{ij}=-$. 
	A product of the form $\beta = p_{i_1 i_2} p_{i_2 i_3}\cdots p_{i_k i_{k+1}},$ 
	where each factor is nonzero and the indices $i_1,i_2,\dots,i_{k+1}$ are all distinct is called a \textit{path} of length $l(\beta)=k$ from $v_{i_1}$ to $v_{i_{k+1}}$.
	If $i_1,i_2,\dots,i_{k}$ are all distinct, then $\gamma=p_{i_1i_2}p_{i_2i_3}\cdots p_{i_ki_1}$ is called a \textit{simple cycle}, denoted by $\gamma=(v_{i_1},v_{i_2},...,v_{i_k})$, of length $l(\gamma)=k$.  For $k=1$, the simple cycle $\gamma=p_{i_1i_1}$ is called a \textit{loop}. The simple cycle $\gamma$ is said to be \textit{positive} (respectively, \textit{negative}) if
	$\sgn(\gamma)=(-1)^{k-1} \, p_{i_1 i_2} p_{i_2 i_3} \cdots p_{i_k i_1}$ is positive (respectively, negative). 
	{Therefore, if $l(\gamma)$ is even and the product $p_{i_1i_2}p_{i_2i_3}\cdots p_{i_ki_1}$ is positive (respectively, negative), then $\gamma$ is a negative (respectively, positive) simple cycle. If $l(\gamma)$ is odd and the product $p_{i_1i_2}p_{i_2i_3}\cdots p_{i_ki_1}$ is positive (respectively, negative), then $\gamma$ is a positive (respectively, negative) simple cycle.} 
	Suppose that $\gamma_i$'s are mutually vertex disjoint simple cycles with $l(\gamma_i)=k_i$, for $i=1,2,...,t$, then $\Gamma=\gamma_1\gamma_2\cdots\gamma_t$ is called a \textit{composite cycle} of length $l(\Gamma)=\sum_{i=1}^{t}k_i$ and  $\sgn(\Gamma)=\prod_{i=1}^t\sgn(\gamma_i)$.
	{Throughout this paper, we assume all the cycles to be simple unless otherwise mentioned.}

	Let $\mathcal{P} = [p_{ij}] \in \mathcal{Q}_n$ be a symmetric sign pattern matrix. The \textit{underlying signed undirected graph} $G(\mathcal{P})$ of $\mathcal{P}$ is the undirected graph with the edge $\{v_i,v_j\}$ positively signed if $p_{ij}\neq 0$ $\text{for all} \ i,j = 1, 2, \dots, n$.
	A \textit{path} $P_k=\{v_{i_1}, v_{i_2}\} \{v_{i_2}, v_{i_3}\} \cdots \{v_{i_k}, v_{i_{k+1}}\}$ of length $l(P_k) = k$ in $G(\mathcal{P})$ is a sequence of $k$ edges of $G(\mathcal{P})$, where the vertices $v_{i_1}, v_{i_2}, \dots, v_{i_{k+1}}$ are all distinct. For $k \geq 3$, if the vertices $v_{i_1}, v_{i_2},..., v_{i_{k}}$ are all distinct, then $\{v_{i_1}, v_{i_2}\} \{v_{i_2}, v_{i_3}\} \cdots \{v_{i_k}, v_{i_{1}}\}$ is called a \textit{simple cycle}, denoted by $C_k = v_{i_1} v_{i_2} \cdots v_{i_k}$, and $l(C_k)=k$. Suppose that $C_{k_i}$'s are mutually vertex disjoint simple cycles in $G(\mathcal{P})$, of length $k_i$, for $i=1,2,...,t$, then $C=C_{k_1}C_{k_2}\cdots C_{k_t}$ is called a \textit{composite cycle} of length $l(C)=\sum_{i=1}^{t}k_i$.  If $S$ is a set of vertices in $G(\mathcal{P})$, then $G(\mathcal{P}) \setminus S$ denotes the subgraph of $G(\mathcal{P})$ obtained by deleting the vertices in $S$ along with their incident edges. For a vertex $v_i$ of $G(\mathcal{P})$, the \textit{degree} of $v_i$, denoted by $\deg(v_i)$, is the number of edges of $G(\mathcal{P})$ incident to $v_i$. A symmetric sign pattern $\mathcal{P}\in \mathcal{Q}_n$ is called a \textit{tree sign pattern} (respectively, \textit{path sign pattern}) if the underlying signed undirected graph $G(\mathcal{P})$ is a tree (respectively, path).

	A set of edges
	$
	\mathcal{M} = \{\{v_{i_1}, v_{j_1}\}, \{v_{i_2}, v_{j_2}\}, \dots, \{v_{i_k}, v_{j_k}\}\}
	$
	in $G(\mathcal{P})$ is called a \textit{matching} if no two edges in $\mathcal{M}$ share a common vertex. The \textit{length} of a matching $\mathcal{M}$, denoted by $l(\mathcal{M})$, is the number of edges in $\mathcal{M}$. A matching $\mathcal{M}$ is said to be a \textit{maximum matching} if $l(\mathcal{M})$ is maximum among all matchings in $G(\mathcal{P})$. It is called a \textit{perfect matching} if every vertex of $G(\mathcal{P})$ is incident with exactly one edge in $\mathcal{M}$. Note that each perfect matching of $G(\mathcal{P})$ corresponds to a nonzero term in the determinant expansion of $\mathcal{P}$.

Let $A\in\mathbb{R}^{n\times n}$. If $\alpha$, $\beta$ are subsets of $\{1,2,...,n\}$, then $A(\alpha,\beta)$ denotes the submatrix of $A$ obtained by deleting rows, columns corresponding to the indices in $\alpha, \beta$, respectively. In particular, $A(\alpha)$ denotes $A(\alpha, \alpha)$. For $\alpha = \{k\}$ with $k \in \{1,2,\ldots,n\}$, we denote $A(\alpha)$ by $A(k)$.

	Let $A=[a_{ij}] \in \mathbb{R}^{n\times n}$ be a real symmetric matrix and let $B$ be a principal submatrix of $A$ of order $n-1$. Suppose that 
$\lambda_1 \ge \lambda_2 \ge \cdots \ge \lambda_n$ and 
$\mu_1 \ge \mu_2 \ge \cdots \ge \mu_{n-1}$ 
are the eigenvalues of $A$ and $B$, respectively. Then, by the Cauchy interlacing theorem, we have
\begin{equation}\label{eqx1}
	\lambda_1\geq\mu_1\geq \lambda_2\geq\cdots
\geq\lambda_i\geq \mu_i\geq \lambda_{i+1}\geq \cdots\geq\lambda_{n-1}\geq \mu_{n-1}\geq \lambda_n.\end{equation}
Let $m_A(\lambda_i)$ denote the algebraic multiplicity of $\lambda_i$ as an eigenvalue of $A$.
From~\eqref{eqx1}, it follows that $|m_{A(j)}(\lambda_i)-m_A(\lambda_i)|\leq 1$, for $1\leq j\leq n$. In \cite{siam}, Johnson et al. define the vertex $v_j$, for $1 \leq j \leq n$, to be a \textit{Parter vertex} of $A$ for $\lambda_i$ if $m_{A(j)}(\lambda_i)-m_A(\lambda_i)=1$. The vertex $v_j$ is called a \textit{$P$-vertex} of $A$ if $m_{A(j)}(0)-m_A(0)=1$ (see, e.g., \cite{2008}). The number of $P$-vertices of $A$ is denoted by $P_\nu(A)$. 

\begin{defn}
	A symmetric sign pattern $\mathcal{P}\in \mathcal{Q}_n$ is said to \textit{require full $P$-vertices} if for every matrix $A \in \mathcal{Q}_{\mathrm{SYM}}(\mathcal{P})$, $P_{\nu}(A)=n$.
\end{defn}

\begin{eg}\label{xxeg2.2} \rm
	Consider the symmetric path sign pattern $\mathcal{P}\in \mathcal{Q}_4$ with a $0$-diagonal, whose underlying signed undirected graph $G$ is given in Fig.\ref{fignx1}.\\
	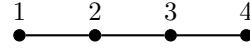
\begin{figure}[H]
		\centering
		\begin{minipage}{.45\textwidth}
			\vspace{-1cm}
			\[
			\mathcal{P}=\begin{bmatrix}
				0 & + & 0 & 0 \\
				+ & 0 & + & 0 \\
				0 & + & 0 & + \\
				0 & 0 & + & 0
			\end{bmatrix}
			\]
		\end{minipage}
		\hspace{0.05\textwidth}
		\begin{minipage}{.45\textwidth}
			\centering
			\tikzset{node distance=1cm}
			\begin{tikzpicture}[
			  vertex/.style={draw, circle, inner sep=0pt, minimum size=.15cm, fill=black},
			edge/.style={thick}
				]
			 \node[vertex,label=above:1] (v1) at (-2,0) {};
			\node[vertex,label=above:2] (v2) at (-1,0) {};
			\node[vertex,label=above:3] (v3) at ( 0,0) {};
			\node[vertex,label=above:4] (v4) at ( 1,0) {};
				
			 \draw[edge] (v1) --  (v2);
			\draw[edge] (v2) -- (v3);
			\draw[edge] (v3) --  (v4);
			\end{tikzpicture}
			\caption{Signed undirected graph \(G\) of \(\mathcal{P}\).}
			\label{fignx1}
		\end{minipage}
		
	\end{figure}
	
 Consider	$$A=\begin{bmatrix}
		0&a&0&0\\ a&0&b&0\\ 0&b&0&c\\0&0&c&0
	\end{bmatrix}\in \mathcal{Q}_{\mathrm{SYM}}(\mathcal{P}), ~\text{where}~ a,b,c>0.$$
Then $\det(A)=a^2c^2$, and hence $m_A(0)=0$. However, for $1\leq i \leq 4$, $\det(A(i))=0$, and therefore $m_{A(i)}(0)=1$. Thus, for $1\leq i \leq 4$, we have
$m_{A(i)}(0)=m_A(0)+1.$
Hence, $\mathcal{P}$ requires full $P$-vertices.

\end{eg}

	Let $A=[a_{ij}] \in \mathbb{R}^{n\times n}$ be a real symmetric matrix, and let $G(A)$ denote its underlying signed undirected graph with vertex set $\{v_1, v_2, \ldots, v_n\}$ and edge set $\{(v_i, v_j) : a_{ij} \neq 0,\ i \neq j\}$.
Since the diagonal entries of $A$ are not taken into account, the graph $G(A)$ is independent of the main diagonal entries of $A$. For a graph $G$ on $n$ vertices, define
$$
S(G)=\{A \in \mathbb{R}^{n\times n} : A = A^{\mathsf{T}},\ G(A)=G\}.
$$

 The study of eigenvalue multiplicities of real symmetric matrices can be traced back to the work of Parter \cite{1960}, who investigated the multiplicities of eigenvalues for matrices whose underlying signed undirected graph is a tree. These results were subsequently extended by Wiener \cite{1984}. For Hermitian matrices, the concept of Parter vertices was introduced by Johnson and Sutton in \cite{siam}. They showed that each singular acyclic matrix of order $n$ has at most $n-2$ $P$-vertices. For real symmetric matrices, the notion of $P$-vertices, as currently understood, was introduced by Kim and Shader in \cite{2008, 2009}. In \cite{2009} Kim and Shader, studied $P$-vertices of nonsingular acyclic matrices whose underlying signed undirected graph is a path or a star. In that work, they also proposed several open problems related to $P$-vertices of acyclic matrices. These problems were subsequently studied in several works (see, e.g., \cite{2011,2013,2013a,2014}). In particular, Anđelić et al. \cite{2013} classified the trees $T$ for which there exists a nonsingular matrix $A \in S(T)$ with full $P$-vertices. Subsequently, Howlader et al. \cite{2025} extended the analysis to unicyclic graphs and characterized those graphs $G$ for which there exists a nonsingular matrix $A \in S(G)$ with full $P$-vertices. Moreover, Sharma et al. \cite{2026} proved that if $G$ is a bipartite graph containing a perfect matching, then there exists a nonsingular matrix $A \in S(G)$ with full $P$-vertices.

	However, characterizing graphs $G$ for which every matrix in $S(G)$ has full $P$-vertices remains unexplored. Since matrices in $S(G)$ with $0$-diagonal are determined by their sign patterns, we study this problem from a sign pattern perspective. In this paper, we consider symmetric sign patterns with a $0$-diagonal and establish necessary and sufficient conditions under which such sign patterns require full $P$-vertices.

    The structure of this paper is as follows. In Section~\ref{s2}, we consider symmetric sign patterns with a $0$-diagonal whose underlying signed undirected graph is a tree. Theorem~\ref{thx1} provides a characterization of such patterns that require full $P$-vertices in terms of the existence of a perfect matching in the underlying signed undirected graph.

   In Section~\ref{s3}, we study symmetric sign patterns whose underlying graph contain cycles but no loops. We focus on sign patterns whose underlying signed undirected graph is unicyclic and derive several necessary and sufficient conditions for such patterns to require full $P$-vertices. In addition, Theorem~\ref{th0.61e} provides a sufficient condition for a nonnegative symmetric sign pattern associated with a dumbbell graph to require full $P$-vertices.

   We conclude with several interesting consequences of the results established in this paper that could not be included in the manuscript. In Section~\ref{s5}, we also pose a number of open problems that arise naturally from our work and provide directions for future research in this area.
	
	\section{Tree Sign Patterns.}\label{s2}
In this section, we study symmetric tree sign patterns with a $0$-diagonal that require full $P$-vertices.

\begin{remark}\cite[Remark 4.1]{2018} \label{l1}\rm
	Let $\mathcal{P}$ be a symmetric sign pattern whose underlying signed undirected graph is a tree. Then for every matrix $A \in \mathcal{Q}(\mathcal{P})$, there exists a diagonal matrix $D$ such that $D^{-1}AD$ is a
	symmetric matrix in $\mathcal{Q}(\mathcal{P})$. Note that $D^{-1}AD$ and $A$ have the same spectrum. Moreover,
	$D$ can be chosen so that $D^{-1}AD$ is symmetric and all nonzero off-diagonal entries are positive.
\end{remark}

By Remark~\ref{l1}, we may assume, without loss of generality, that $\mathcal{P}$ is a nonnegative tree sign pattern with a $0$-diagonal. 
	Let $A=[a_{ij}] \in \mathbb{R}^{n\times n}$ be a real symmetric matrix. The \textit{inertia} of $A$, denoted by $\In(A)$, is defined as the triple $\In(A)=(i_+(A), i_-(A), i_0(A))$, where $i_+(A)$, $i_-(A)$ and $i_0(A)$ denote the number of eigenvalues of $A$ with positive, negative and zero real parts, respectively. For a sign pattern $\mathcal{P} \in \mathcal{Q}_n$, the inertia of $\mathcal{P}$ is defined as $\In(\mathcal{P})=\{\In(A): A\in \mathcal{Q}(\mathcal{P})\}$. We first study the inertia of $\mathcal{P}$ whose underlying signed undirected graph is a path.
	
\begin{prop}\cite[Proposition 3.1]{2001a} \label{1.11pp}\rm
	For the $n\times n$ symmetric tri-diagonal pattern
	\[
	\mathcal{P} = \begin{pmatrix}
		0 & + & 0 & \dots  & \dots     &  0 \\
		+ & 0 & +  & 0 & \dots     & 0   \\
		0 & + & 0  & +     & \ddots    & \vdots    \\
		\vdots&  \vdots & \ddots    &   \ddots   & \ddots &\vdots  \\
		0& 0   &  \dots         &     + &0      & +\\
		0& 0   &  \dots         &     0 &+      & 0
	\end{pmatrix},
	\]
	\begin{itemize}
		\item[(a)]  if $n$ is even, then $\mathcal{P}$ is sign nonsingular, and $\In(\mathcal{P})=\{(\frac{n}{2}, \frac{n}{2},0)\}$,
		
		\item[(b)] if $n$ is odd, then $\mathcal{P}$ is sign singular, and $\In(\mathcal{P})=\{(\frac{n-1}{2}, \frac{n-1}{2},1)\}$.
	\end{itemize}
\end{prop}	
	
The following theorem characterizes symmetric tree sign patterns that require full $P$-vertices in terms of perfect matchings in their underlying signed undirected graphs.

	\begin{theorem}\label{thx1}
		Let $\mathcal{P}\in \mathcal{Q}_n$ be a symmetric tree sign pattern with a $0$-diagonal. Then $\mathcal{P}$ requires full $P$-vertices if and only if the underlying signed undirected graph $G(\mathcal{P})$ has a perfect matching.
	\end{theorem}
	\begin{proof}
		First, suppose that $G(\mathcal{P})$ has a perfect matching. Then $n$ must be even. 
		Since $G(\mathcal{P})$ is a tree, the perfect matching is unique. 
	   Therefore, for all $A\in \mathcal{Q}_{\mathrm{SYM}}(\mathcal{P})$, $\det(A)\neq 0$, and hence $m_A(0)=0$.
		For $i\in\{1,\ldots,n\}$, $G(\mathcal{P})\setminus\{v_i\}$ is a tree of odd order and hence it does not have a perfect matching. 
		Hence $\det(A(i))=0$, since $|m_{A(i)}(0)-m_A(0)|\leq 1$, which implies that
		$m_{A(i)}(0)=m_A(0)+1$, for all $i=1,2,\ldots,n$. 
		Therefore, $\mathcal{P}$ requires full $P$-vertices.
		
		For the converse, suppose that $G(\mathcal{P})$ does not have a perfect matching. 
		Let $\mathcal{M}$ be a maximum matching in $G(\mathcal{P})$ of length $m$. 
		Choose $A\in \mathcal{Q}_{\mathrm{SYM}}(\mathcal{P})$ by symmetrically emphasizing the entries corresponding to the edges in $\mathcal{M}$, then $\rank(A)\geq 2m$. Also, the maximum length of composite cycles in the signed directed graph of $\mathcal{P}$ is $2m$, and hence $\operatorname{rank}(A)\leq 2m$. Therefore, $\operatorname{rank}(A)=2m$, and thus $m_A(0)=n-2m$.
		Let $v_i\in V(G(\mathcal{P}))\setminus V(\mathcal{M})$. Then $\mathcal{M}$ is also a maximum matching in $G(\mathcal{P})\setminus \{v_i\}$. 
		Hence $m_{A(i)}(0)=(n-1)-2m$.
		Therefore, $v_i$ is not a $P$-vertex of $A$.
	\end{proof}
	
	The following corollary for path sign patterns follows immediately from Theorem~\ref{thx1}.
	
	\begin{cor}\label{th0.1e}
		Suppose that $\mathcal{P}\in \mathcal{Q}_n$ is a symmetric tridiagonal sign pattern with a $0$-diagonal. Then $\mathcal{P}$ requires full $P$-vertices if and only if $n$ is even.
	\end{cor}

	\section{Sign Patterns Whose Underlying Signed Undirected Graph has Cycles.}\label{s3}
	
In this section, we consider symmetric sign patterns whose underlying signed undirected graph has cycles. We begin with those whose underlying signed undirected graph is a cycle.
	
\begin{theorem} \label{th0.2e}
	Let $\mathcal{P} \in \mathcal{Q}_n$ be a symmetric sign pattern with a $0$-diagonal whose underlying signed undirected graph is a cycle of length $n$. 
	Then $\mathcal{P}$ requires full $P$-vertices if and only if $n$ is even and 
	every composite cycle of length $n$ in the signed directed graph $D(\mathcal{P})$ of $\mathcal{P}$
	has the same sign. 
\end{theorem}
\begin{proof}
Suppose that $n$ is even and every composite cycle of length $n$ in $D(\mathcal{P})$ has the same sign. Then, for all $A\in \mathcal{Q}_{\mathrm{SYM}}(\mathcal{P})$, we have $\det(A)\neq 0$, and hence $m_A(0)=0$. 
For $i\in \{1,2,\ldots,n\}$, the matrix $A(i)$ is a symmetric tridiagonal matrix with a $0$-diagonal of odd order. By Proposition~\ref{1.11pp}(b), we have $m_{A(i)}(0)=1$. 
Therefore, $m_{A(i)}(0)=m_A(0)+1$ for all $i=1,2,\ldots,n$,
and hence $P_{\nu}(A)=n$ for all $A\in\mathcal{Q}_{\mathrm{SYM}}(\mathcal{P})$.
Thus, $\mathcal{P}$ requires full $P$-vertices.
	
	For the converse, first suppose that $n$ is even and $D(\mathcal{P})$ contains composite cycles of length $n$ with opposite signs.
	Then there exists $A\in \mathcal{Q}_{\mathrm{SYM}}(\mathcal{P})$ such that $\det(A)=0$, and hence $m_A(0)\geq 1$. For $i\in \{1,2,\ldots,n\}$, $A(i)$ is a symmetric tridiagonal matrix of odd order with a $0$-diagonal. By Proposition~\ref{1.11pp}(b), $m_{A(i)}(0)=1$. Consequently, $m_{A(i)}(0)\neq m_A(0)+1$ for all $i\in \{1,2,\ldots,n\}$, and hence $\mathcal{P}$ does not require full $P$-vertices.
	
	If $n$ is odd, then for every $A=[a_{ij}]\in \mathcal{Q}_{\mathrm{SYM}}(\mathcal{P})$, $\det(A)=2a_{12}a_{23}\cdots a_{n1}$.
	Thus $A$ is nonsingular. Moreover, for each $i\in \{1,2,\ldots,n\}$, the matrix $A(i)$ is a symmetric tridiagonal matrix with a $0$-diagonal of even order. By Proposition~\ref{1.11pp}(a), $A(i)$ is also nonsingular. Hence $\mathcal{P}$ does not require full $P$-vertices.
	
	Therefore, $\mathcal{P}$ requires full $P$-vertices if and only if $n$ is even and 
	every cycle of length $n$ in the signed directed graph $D(\mathcal{P})$
\end{proof}	
	
	\begin{eg}\rm
		Consider the symmetric sign pattern $\mathcal{P}\in \mathcal{Q}_{4}$ with a $0$-diagonal whose derlying signed undirected graph is $D(\mathcal{P})$, as shown in Figure~\ref{figx11}.\\
			\begin{figure}[H]
			\centering
			\begin{minipage}{.45\textwidth}
				\vspace{-1cm}
				\[
				\mathcal{P}=\begin{bmatrix}
					0 & + & 0 & - \\
					+ & 0 & + & 0 \\
					0 & + & 0 & + \\
					- & 0 & + & 0
				\end{bmatrix}
				\]
			\end{minipage}
			\hspace{0.05\textwidth}
			\begin{minipage}{.45\textwidth}
		\centering
		\usetikzlibrary{decorations.markings}
		
		\tikzset{
			vertex/.style={
				draw, circle, fill=black,
				inner sep=0pt, minimum size=0.15cm
			},
			edge/.style={
				thick,
				postaction={
					decorate,
					decoration={
						markings,
						mark=at position 0.75 with {\arrow{stealth}}
					}
				}
			}
		}
		
		\begin{tikzpicture}[scale=1, every node/.style={font=\small}]
				\node[vertex,label=above:$u_1$] (v1) at (-0.75, 0.75) {};
				\node[vertex,label=above:$u_2$] (v2) at ( 0.75, 0.75) {};
				\node[vertex,label=below:$u_3$] (v3) at ( 0.75,-0.75) {};
				\node[vertex,label=below:$u_4$] (v4) at (-0.75,-0.75) {};
				
	\draw[edge] (v1) -- node[below] {$+$} (v2);
	\draw[edge, bend right=25] (v2) to node[above] {$+$} (v1);			
	\draw[edge] (v2) -- node[left] {$+$} (v3);
	\draw[edge, bend right=25] (v3) to node[right] {$+$} (v2);
	\draw[edge] (v3) -- node[above] {$+$} (v4);
	\draw[edge, bend right=25] (v4) to node[below] {$+$} (v3);
	\draw[edge] (v4) -- node[right] {$-$} (v1);
	\draw[edge, bend right=25] (v1) to node[left] {$-$} (v4);
			\end{tikzpicture}
			\caption{$D(\mathcal{P})$.}
				\label{figx11}
			\end{minipage}
			\end{figure}
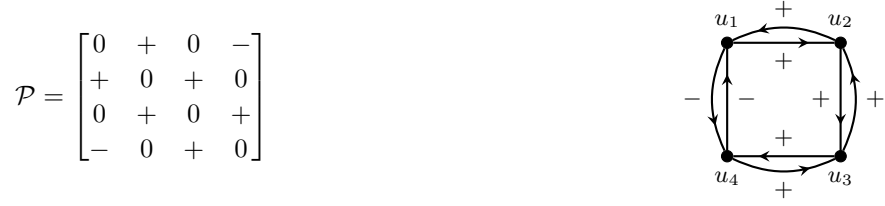
	The composite cycles in $D(\mathcal{P})$ of length $4$ are $(u_1,u_2,u_3,u_4)$, $(u_1,u_4,u_3,u_2)$, $(u_1,u_2)(u_3,u_4)$ and $(u_1,u_4)(u_2,u_3)$, and all these cycles are positively signed. Therefore, $\mathcal{P}$ requires full $P$-vertices.
	\end{eg}

    The following definition follows from Definition~3.1 in \cite{2025}.
	
	\begin{defn}\rm
		Let $\mathcal{P} \in \mathcal{Q}_n$ be a symmetric sign pattern, and $P_m$ be a path sign pattern of length $m$ with a $0$-diagonal. The sign pattern $\mathcal{P} \cdot P_m$ is defined to be the sign pattern whose underlying signed undirected graph is obtained by attaching the path $G(P_m)$ to a vertex of $G(\mathcal{P})$.
	\end{defn}
	
	\begin{eg}\rm
Consider the symmetric sign pattern $\mathcal{P}\in \mathcal{Q}_{7}$ with a $0$-diagonal whose underlying signed undirected graph is $C_6 \cdot G(P_1)$, as shown in Figure~\ref{figx1.1}.
		\begin{figure}[H]
			\centering
			\tikzset{node distance=1cm}
		\begin{tikzpicture}[
			vertex/.style={draw, circle, inner sep=0pt, minimum size=.15cm, fill=black},
			edge/.style={thick}
			]
			\node[vertex] (v0) at (0,0) {};
			\node[vertex] (v1) at (-3,0) {};
			\node[vertex] (v2) at (-2,1) {};
			\node[vertex] (v3) at (-1,1) {};
			\node[vertex] (v4) at (-1,-1) {};	 
			\node[vertex] (v5) at (-2,-1) {};	
			\node[vertex] (v6) at (-4,0) {};
			
			\draw[edge] (v3)--(v0);
			\draw[edge] (v4)--(v0);
			\draw[edge] (v1)--(v2);
			\draw[edge] (v2)--(v3);
			\draw[edge] (v4)--(v5);
			\draw[edge] (v5)--(v1);
			\draw[edge] (v1)--(v6);

		\end{tikzpicture}
		\caption{$C_6 \cdot G(P_1)$} \label{figx1.1}
	\end{figure}
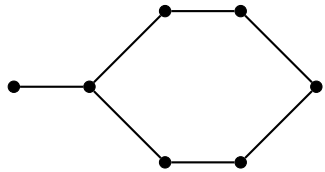
	\end{eg}

\begin{theorem}\label{th1.1}
	Suppose that $\mathcal{P}\in \mathcal{Q}_n$ is a symmetric sign pattern with a $0$-diagonal. Let $G(\mathcal{P}\cdot P_1)$ be obtained from $G(\mathcal{P})$ by attaching $G(P_1)$ to the vertex $u$ of $G(\mathcal{P})$. Then $\mathcal{P}\cdot P_1$ requires full $P$-vertices implies $\mathcal{P}'$ requires full $P$-vertices, where $\mathcal{P}'$ is obtained from $\mathcal{P}$ by deleting the rows and columns corresponding to $u$.
\end{theorem}

\begin{proof}
	Suppose that $G(P_1)$ has a vertex set $\{v_1, v_2\}$ and $G(\mathcal{P})$ has a vertex set $\{v_2, v_3, \ldots, v_{n+1}\}$, where $u = v_2$. Since $\mathcal{P}\cdot P_1$ requires full $P$-vertices, therefore $P_\nu (A)=n$, for all $A  \in \mathcal{Q}_{\mathrm{SYM}}(\mathcal{P}\cdot P_1)$. Let $B\in \mathcal{Q}_{\mathrm{SYM}}(\mathcal{P}')$. Consider,
	\[
	A =
	\begin{bmatrix}
		0 & a_{12} & \mathbf{0}^{\mathsf T} \\
		a_{12} & 0 & C^{\mathsf T} \\
		\mathbf{0} & C & B
	\end{bmatrix}\in \mathcal{Q}_{\mathrm{SYM}}(\mathcal{P}\cdot P_1),
	\]
	where $\mathbf{0}$ denotes the zero vector of order $n-1$ and $C$ is a vector of order $n-1$.
	Since
	\begin{equation}\label{eq1.2}
		\begin{bmatrix}
			0 & a_{12} & \mathbf{0}^{\mathsf T} \\
			a_{12} & 0 & \mathbf{0}^{\mathsf T} \\
			\mathbf{0} & \mathbf{0} & B
		\end{bmatrix}
		=
		\begin{bmatrix}
			1 & 0 & \mathbf{0}^{\mathsf T} \\
			0 & 1 & \mathbf{0}^{\mathsf T} \\
			- a_{12}^{-1} C & \mathbf{0} & I_{n-1}
		\end{bmatrix}
		\begin{bmatrix}
			0 & a_{12} & \mathbf{0}^{\mathsf T} \\
			a_{12} & 0 & C^{\mathsf T} \\
			\mathbf{0} & C & B
		\end{bmatrix}
		\begin{bmatrix}
			1 & 0 & -a_{12}^{-1} C^{\mathsf T} \\
			0 & 1 & \mathbf{0}^{\mathsf T} \\
			\mathbf{0} & \mathbf{0} & I_{n-1}
		\end{bmatrix},
	\end{equation}
	it follows that $A$ is congruent to
	\[
	\begin{bmatrix}
		0 & a_{12} \\
		a_{12} & 0
	\end{bmatrix}
	\oplus B.
	\]
	For any $i$ with $3 \le i \le n+1$, by an argument similar to that in~\eqref{eq1.2}, $A(i)$ is congruent to
	\[
	\begin{bmatrix}
		0 & a_{12} \\
		a_{12} & 0
	\end{bmatrix}
	\oplus B(i-2).
	\]
	Since $v_i$ is a $P$-vertex of $A$, we have
	$m_{A(i)}(0) = m_A(0) + 1$.
	Consequently, $m_{B(i-2)}(0) = m_B(0) + 1$, for all $i \in \{3,4,\ldots,n+1\}$. Therefore $P_\nu (B)=n-1$. Since $B\in \mathcal{Q}_{\mathrm{SYM}}(\mathcal{P'})$ is arbitrary, $\mathcal{P}'$ requires full $P$-vertices.
\end{proof}

The converse of the theorem may not be true.

\begin{eg}\rm
Consider the symmetric sign patterns $\mathcal{P}\cdot P_1$, $\mathcal{P}'$ whose underlying signed undirected graphs are given in Figure~\ref{fig1.1}.
\begin{figure}[H]
	\centering
	\tikzset{node distance=1cm}
	
	\begin{tikzpicture}[
		vertex/.style={draw, circle, inner sep=0pt, minimum size=.15cm, fill=black},
		edge/.style={thick}
		]
		\node[vertex, label=above:$v_2$] (v1) at (-3,0) {};
		\node[vertex, label=above:$v_3$] (v2) at (-2,1) {};
		\node[vertex, label=above:$v_4$] (v3) at (-1,1) {};
	\node[vertex, label=below:$v_5$] (v4) at (-1,-1) {};	 
	\node[vertex, label=below:$v_6$] (v5) at (-2,-1) {};	
		\node[vertex, label=above:$v_1$] (v6) at (-4,0) {};
		
		\draw[edge] (v1)--(v2);
		\draw[edge] (v2)--(v3);
		\draw[edge] (v3)--(v4);
		\draw[edge] (v4)--(v5);
		\draw[edge] (v5)--(v1);
		\draw[edge] (v1)--(v6);
		
		\draw[very thick,->] (0,0) -- (1,0);
		
		vertex/.style={draw, circle, inner sep=0pt, minimum size=.15cm, fill=black},
		edge/.style={thick}
		]
\node[vertex, label=above:$v_3$] (v1) at (2,0) {};
		\node[vertex, label=above:$v_4$] (v2) at (3,0) {};
		\node[vertex, label=above:$v_5$] (v3) at (4,0) {};
		\node[vertex, label=above:$v_6$] (v4) at (5,0) {};
		
		\draw[edge] (v1)--(v2);
		\draw[edge] (v2)--(v3);
		\draw[edge] (v4)--(v3);
	\end{tikzpicture}
	\caption{Underlying signed undirected graphs of $\mathcal{P}\cdot P_1$, $\mathcal{P}'$, respectively.} \label{fig1.1}
\end{figure}
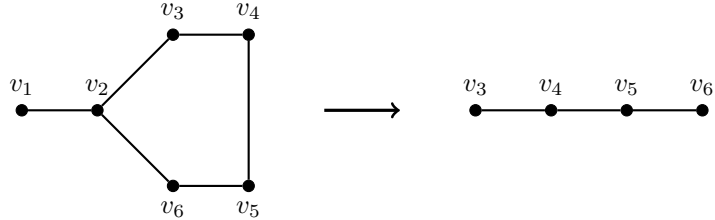
Then the signed directed graph 
$
D(\mathcal{P}\cdot P_1)
$
has exactly one composite cycle, namely 
$
(v_1,v_2)(v_3,v_4)(v_5,v_6),
$
of length $6$. Moreover, $G(\mathcal{P})$ is a cycle of odd length. Hence, for every $A \in \mathcal{Q}_{\mathrm{SYM}}(\mathcal{P}\cdot P_1)$, both $A$ and $A(1)$ are nonsingular. Consequently, $v_1$ is not a $P$-vertex of $A$, and therefore, $\mathcal{P}\cdot P_1$ does not require full $P$-vertices. 
However, by Corollary~\ref{th0.1e}, $\mathcal{P}' = P_3$ requires full $P$-vertices.
\end{eg}

Next, consider the case of attaching a path of length two.

\begin{theorem} \label{th0.4e}
	Suppose that $\mathcal{P}\in \mathcal{Q}_n$ is a symmetric sign pattern with a $0$-diagonal. Then  $\mathcal{P}\cdot P_2$ requires full $P$-vertices if and only if $\mathcal{P}$ requires full $P$-vertices.
\end{theorem}

\begin{proof}  
	Suppose that $G(P_2)$ has vertex set $\{v_1, v_2, v_3\}$ and $G(\mathcal{P})$ has vertex set $\{v_3, v_4, \ldots, v_{n+2}\}$. Since $\mathcal{P}\cdot P_2=(\mathcal{P}\cdot P_1)\cdot P_1$, where $G(P_1)$ intersects with $G(\mathcal{P}\cdot P_1)$ at $v_2$. Hence, by Theorem~\ref{th1.1}, $\mathcal{P}\cdot P_2$
	requires full $P$-vertices implies that $\mathcal{P}$ requires full $P$-vertices.
	
	Conversely, suppose that $\mathcal{P}$ requires full $P$-vertices. Consider $\mathcal{P}\cdot P_2=[p_{ij}] \in \mathcal{Q}_{n+2}$, and let $A \in \mathcal{Q}_{\mathrm{SYM}}(\mathcal{P}\cdot P_2)$. Then $A$ can be written as
	\[
	A=\left[
	\begin{array}{cc|cccc}
		0 & a & 0 & 0 & \cdots & 0\\
		a & 0 & b & 0 & \cdots & 0\\ \hline
		0 & b &   &   &        &   \\
		0 & 0 &   &   &        &   \\
		\vdots & \vdots & &  & B & \\
		0 & 0 &   &   &        &
	\end{array}\right]
	= \left[
	\begin{array}{cc|c}
		0 & a & \mathbf{0}^{\mathsf T}\\
		a & 0 & C^{\mathsf T}\\ \hline
		\mathbf{0} & C & B
	\end{array}\right],
	\] where $B=[b_{ij}] \in \mathcal{Q}_\mathrm{SYM}(\mathcal{P})$, and $a$ and $b$ are nonzero real numbers such that $\sgn(a)=p_{12}$ and $\sgn(b)=p_{23}$.
	Then, by an argument similar to that in~\eqref{eq1.2}, $A$ is congruent to
	\begin{equation}\label{eq1.23}
		\begin{bmatrix}
			0 & a\\
			a & 0
		\end{bmatrix}
		\oplus B.    
	\end{equation}

	For $i=1$, \[
	A(1)=\left[
	\begin{array}{cc|cccc}
		0 & b & 0 & 0 & \cdots & 0\\
		b & b_{11} & b_{12} & b_{13} & \cdots & b_{1n}\\ \hline
		0 & b_{21} &   &   &        &   \\
		0 & b_{31} &   &   &        &   \\
		\vdots & \vdots & &  & B(1) & \\
		0 & b_{n1} &   &   &        &
	\end{array}\right]
	= \left[
	\begin{array}{cc|c}
		0 & b & \mathbf{0}^{\mathsf T}\\
		b & b_{11} & C_{1}^{\mathsf T}\\ \hline
		\mathbf{0} & C_{1} & B(1)
	\end{array}\right],
	\]
	\begin{equation}\label{eq1.22}
		\begin{bmatrix}
			0 & b & \mathbf{0}^{\mathsf T} \\
			b & b_{11} & \mathbf{0}^{\mathsf T} \\
			\mathbf{0} & \mathbf{0} & B(1)
		\end{bmatrix}
		=
		\begin{bmatrix}
			1 & 0 & \mathbf{0}^{\mathsf T} \\
			0 & 1 & \mathbf{0}^{\mathsf T} \\
			-b^{-1} C_{1} & \mathbf{0} & I_{n-1}
		\end{bmatrix}
		\begin{bmatrix}
			0 & b & \mathbf{0}^{\mathsf T} \\
			b & b_{11} & C_1^{\mathsf T} \\
			\mathbf{0} & C_1 & B(1)
		\end{bmatrix}
		\begin{bmatrix}
			1 & 0 & -b^{-1} C_1^{\mathsf T} \\
			0 & 1 & \mathbf{0}^{\mathsf T} \\
			\mathbf{0} & \mathbf{0} & I_{n-1}
		\end{bmatrix},
	\end{equation}    
	Therefore,  $A(1)$ is congruent to
	\[
	\begin{bmatrix}
		0 & b \\
		b & b_{11}
	\end{bmatrix}
	\oplus B(1).
	\]
	Since $v_3$ is a $P$-vertex of $B$, it follows that $
	m_{B(1)}(0) = m_B(0) + 1$.
	Also, both $\begin{bmatrix}
		0&a\\a&0
	\end{bmatrix}~\text{and}~\begin{bmatrix}
		0&b\\b&b_{11}
	\end{bmatrix}$ are nonsingular, hence from \eqref{eq1.23},
	$m_{A(1)}(0) = m_A(0) + 1.$  Therefore, $v_1$ is a $P$-vertex of $A$.
	
	For $i=2$, we have $A(2) = [0]\oplus B$. Hence, from~\eqref{eq1.23}, it follows that
	$m_{A(2)}(0) = m_A(0) + 1.$  Therefore, $v_2$ is a $P$-vertex of $A$.
	
	For every $3 \le i \le n+2$, by an argument similar to that in~\eqref{eq1.2}, the matrix
	$A(i)$ is congruent to
	$
	\begin{bmatrix}
		0 & a\\
		a & 0
	\end{bmatrix}
	\oplus B(i-2).
	$
	Since $v_i$ is a $P$-vertex of $B$, it follows that $
	m_{B(i-2)}(0) = m_B(0) + 1$. Hence from \eqref{eq1.23},
	$
	m_{A(i)}(0) = m_A(0) + 1$. Therefore, $v_i$ is a $P$-vertex of $A$, for $3\le i\le n+2$.
	
	Hence, $\mathcal{P}\cdot P_2$ requires full $P$-vertices.
\end{proof}

We now establish an interesting consequence of Theorem~\ref{th0.4e}.

\begin{theorem}\label{th0.6x}
	Let $\mathcal{P} \in \mathcal{Q}_n$ and $\mathcal{P}' \in \mathcal{Q}_m$ be two symmetric sign patterns with a $0$-diagonal, such that $\mathcal{P} = \mathcal{P}' \cdot P_{2k}$, for some $k \in \mathbb{N}$. Then $\mathcal{P}$ requires full $P$-vertices if and only if $\mathcal{P}'$ requires full $P$-vertices.
\end{theorem}

\begin{proof}
Suppose that $G(P_{2k})=\{v_0,v_1\}\{v_1,v_2\}\cdots\{v_{2k-1},v_{2k}\}$ is the path from a vertex $v_0$ of $G(\mathcal{P}')$ to $v_{2k}$, where none of $v_1,v_2,\ldots,v_{2k}$ belongs to $V(G(\mathcal{P}'))$ (see Figure \ref{fig5.a}).

	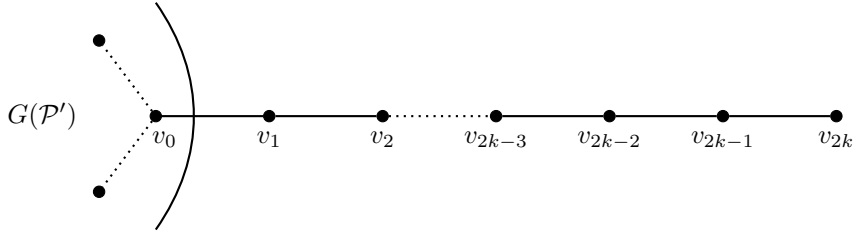
\begin{figure}[H]
	\tikzset{node distance=1cm}
	\centering
	\begin{tikzpicture}
		\tikzstyle{vertex}=[draw, circle, inner sep=0pt, minimum size=.15cm, fill=black]
		\tikzstyle{edge}=[,thick]
		\node[vertex,label=below:$~~v_0$](v1)at(0,0){};
		\node[vertex,label=below:$v_1$](v7)at(1.5,0){};
		\node[vertex,label=below:$v_2$](v8)at(3,0){};
		\node[vertex,label=below:$v_{2k-3}$](v9)at(4.5,0){};
		\node[vertex,label=below:$v_{2k-2}$](v10)at(6,0){};
		\node[vertex,label=below:$v_{2k-1}$](v11)at(7.5,0){};
	\node[vertex,label=below:$v_{2k}$](v12)at(9,0){};
		\node[vertex](v13)at(-0.75,1){};
		\node[vertex](v14)at(-0.75,-1){};
	\node at (-1.5,0) {$G(\mathcal{P}')$};
	
		\draw[thick]   (0,1.5) to [bend left=35] node[left] {$ $} (0,-1.5);
		\draw[edge](v1)--node[]{$ $}(v7);
		\draw[edge](v7)--node[]{$ $}(v8);
		\draw[thick, dotted](v8)--node[]{$ $}(v9);
		\draw[edge](v9)--node[]{$ $}(v10);
		\draw[edge](v10)--node[]{$ $}(v11);
		\draw[edge](v12)--node[]{$ $}(v11);
	
		\draw[thick, dotted](v1)--node[]{$ $}(v13);
		\draw[thick, dotted](v1)--node[]{$ $}(v14);	
	\end{tikzpicture}
	\caption{$G(\mathcal{P})$}	
    \label{fig5.a}
\end{figure}
	Therefore,
	$
	G(\mathcal{P}) = G_1 \cdot \{v_{2k-2},v_{2k-1}\} \{v_{2k-1},v_{2k}\},
	$
	where $G_1 = G(\mathcal{P}) \setminus \{v_{2k-1}, v_{2k}\}$. By Theorem~\ref{th0.4e}, it follows that $\mathcal{P}$ requires full $P$-vertices if and only if $\mathcal{P}_1$ requires full $P$-vertices, where $\mathcal{P}_1$ is obtained from $\mathcal{P}$ by deleting the rows and columns corresponding to the vertices $v_{2k-1}$ and $v_{2k}$.
	
	Since the length of $P_{2k}$ is $2k$, continuing this process iteratively, after $k$ steps we obtain that $\mathcal{P}$ requires full $P$-vertices if and only if $\mathcal{P}'$ requires full $P$-vertices.
\end{proof}

Next, we consider the problem of attaching the path $P_k$ to $m$ distinct vertices of a cycle $C_n$ and determine when the resulting graph requires full $P$-vertices. 

\begin{defn}\cite[Definition~2.10]{2025}\rm~
Let $G \cdot mP_k$ denote the graph obtained by attaching a copy of $P_k$ to each of $m$ distinct vertices of $G$. These $m$ selected vertices of $G$ are called hosts.
\end{defn}

The following corollary is a consequence of successive applications of Theorem~\ref{th0.6x} together with Theorem~\ref{th0.2e}.
\begin{cor} \label{cor1x}
	Let $\mathcal{P}$ be a symmetric sign pattern with a $0$-diagonal, whose underlying signed undirected graph is $C_n\cdot mP_{2k}$, for $1\leq m\leq n$.
	If $n$ is even, and every composite cycle in the signed directed graph $D(\mathcal{P})$ of length $n$ containing the vertices of $C_n$ has the same sign, then $\mathcal{P}$ requires full $P$-vertices.
\end{cor}	

In the following auxiliary result, we consider attaching a tree to a vertex of a cycle $C_n$ and determine when the resulting graph requires full $P$-vertices.

\begin{defn}\rm
	Suppose that $\mathcal{P}$ is a symmetric sign pattern whose underlying signed undirected graph is $G(\mathcal{P})$. 
	The distance between two vertices $u,v \in V(G(\mathcal{P}))$, denoted by $\dist(u,v)$, is defined as the minimum length among all paths from $u$ to $v$ in $G(\mathcal{P})$. 
	Furthermore, let $C_k = u_1u_2\cdots u_k$ be a cycle in $G(\mathcal{P})$. Then, the distance from a vertex $u \in V(G(\mathcal{P}))$ to the cycle $C_k$ is defined by
		\[
		\dist(u, C_k) = \min\{\dist(u, u_i) : i = 1,2,\ldots,k\}.
		\]
\end{defn}

\begin{theorem} \label{th0.6e}
	Let $\mathcal{P}$ be a symmetric sign pattern with a $0$-diagonal, whose underlying signed undirected graph is $G(\mathcal{P})$ and signed directed graph is $D(\mathcal{P})$.
	Suppose that $G(\mathcal{P})$ contains exactly one cycle $C$ which is of even length, and every composite cycle in $D(\mathcal{P})$ of length $l(C)$ and containing the vertices of $C$ has the same sign.
	If the distance from $C$ to every leaf of $G(\mathcal{P})$ is even, and for any vertex $u \in V(G(\mathcal{P})) \setminus V(C)$ with $\dist(u,C)$ odd, has $\deg(u)=2$ in $G(\mathcal{P})$, then $\mathcal{P}$ requires full $P$-vertices.
\end{theorem}	

\begin{proof}
	If $G(\mathcal{P})$ has no leaf, then by Theorem~\ref{th0.2e}, 
	$\mathcal{P}$ requires full $P$-vertices. 
	Suppose that $G(\mathcal{P})$ has at least one leaf. We prove the result by induction on $m$, the number of leaves of $G(\mathcal{P})$.
	For $m=1$, we have $G(\mathcal{P})=C\cdot P_{2k}$,
	for some $k\in\mathbb{N}$. By applying Theorem~\ref{th0.6x}, we obtain that $\mathcal{P}$ requires full $P$-vertices if and only if $C$ requires full $P$-vertices. 
	Therefore, by Theorem~\ref{th0.2e}, $\mathcal{P}$ requires full $P$-vertices.
	
	Suppose that the result holds for all $m \le t-1$, where $t \ge 2$. For $m=t$, let $u$ be a leaf of $G(\mathcal{P})$, then $\dist(u,C)$ is even. Suppose that $R=\{v_0,v_1\}\{v_1,v_2\}\cdots\{v_{2r},v\}\{v,u\}$ is a path from a vertex $v_0$ of $C$ to $u$, where none of $v_1,v_2,\ldots,v_{2r},v$ belongs to $V(C)$. 
	Since $\dist(v,C)$ is odd, it follows from the hypothesis that $\deg(v)=2$ in $G(\mathcal{P})$.
	
	Therefore, $G(\mathcal{P}) = G_1 \cdot \{v_{2r},v\}\{v,u\}$,
	where $G_1 = G(\mathcal{P}) \setminus \{v,u\}$. By Theorem~\ref{th0.4e}, it follows that $\mathcal{P}$ requires full $P$-vertices if and only if $\mathcal{P}_1$ requires full $P$-vertices, where $\mathcal{P}_1$ is obtained from $\mathcal{P}$ by deleting the rows and columns corresponding to the vertices $u$ and $v$ (see Figure \ref{fig6.a}).

	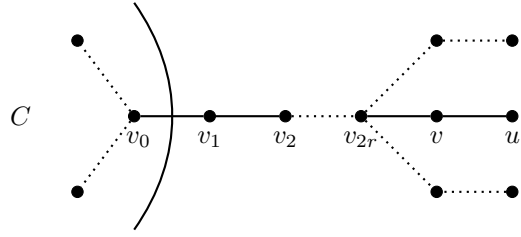
\begin{figure}[H]
		\tikzset{node distance=1cm}
		\centering
		\begin{tikzpicture}
			\tikzstyle{vertex}=[draw, circle, inner sep=0pt, minimum size=.15cm, fill=black]
			\tikzstyle{edge}=[,thick]
			\node at (-1.5,0) {$C$};
			\node[vertex,label=below:$~v_0$](v1)at(0,0){};
			\node[vertex,label=below:$v_1$](v7)at(1,0){};
			\node[vertex,label=below:$v_2$](v8)at(2,0){};
			\node[vertex,label=below:$v_{2r}$](v9)at(3,0){};
			\node[vertex,label=below:$v$](v10)at(4,0){};
			\node[vertex,label=below:$u$](v11)at(5,0){};
			\node[vertex](v12)at(5,1){};
			\node[vertex](v13)at(-0.75,1){};
			\node[vertex](v14)at(-0.75,-1){};
			\node[vertex] (v16) at (4,1) {};
			\node[vertex](v17)at(5,-1){};
			\node[vertex] (v18) at (4,-1) {};
			\draw[thick]   (0,1.5) to [bend left=35] node[left] {$ $} (0,-1.5);
			\draw[edge](v1)--node[]{$ $}(v7);
			\draw[edge](v7)--node[]{$ $}(v8);
			\draw[thick, dotted](v8)--node[]{$ $}(v9);
			\draw[edge](v9)--node[]{$ $}(v10);
			\draw[edge](v10)--node[]{$ $}(v11);
			\draw[thick, dotted](v16)--node[]{$ $}(v12);
			\draw[thick, dotted](v1)--node[]{$ $}(v13);
			\draw[thick, dotted](v1)--node[]{$ $}(v14);
			\draw[thick, dotted](v9)--node[]{$ $}(v16);
			\draw[thick, dotted](v9)--node[]{$ $}(v18);
			\draw[thick, dotted](v17)--node[]{$ $}(v18);
			
		\end{tikzpicture}
		\caption{$G(\mathcal{P})$}
        \label{fig6.a}
	\end{figure}

	If $v_{2r}$ is not a leaf of $G_1$, then $G_1$ has $t-1$ leaves. Hence, by the induction hypothesis, the result follows.

	If $v_{2r}$ is a leaf of $G_1$, then the previous argument applies with $G(\mathcal{P})$ replaced by $G_1$ and $u$ replaced by $v_{2r}$. 
	After finitely many steps, the above process of deleting edges gives a signed undirected graph $G_s$ with fewer than $t$ leaves. Hence, by the induction hypothesis, the result follows.
\end{proof}

The following examples show that the conditions in Theorem~\ref{th0.6e} are necessary.
If there exists a vertex $u \in V(G(\mathcal{P})) \setminus V(\mathcal{C})$ such that 
$\dist(u,\mathcal{C})$ is odd and $\deg(u) \neq 2$, then $\mathcal{P}$ may not require full $P$-vertices.

\begin{eg}\rm
	Let $\mathcal{P}\in \mathcal{Q}_{7}$ be a symmetric sign pattern with a $0$-diagonal, whose underlying signed undirected graph $G(\mathcal{P})$ is  given in Fig.\ref{xnfig4.1a.},\\
		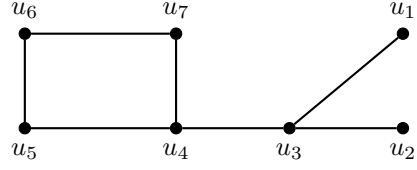
\begin{figure}[H]
		\centering
		\begin{minipage}{.45\textwidth}
			\vspace{-1cm}
			\[
			\mathcal{P}=\begin{bmatrix}
				0 & 0 & + & 0 & 0 & 0 & 0 \\
				0 & 0 & + & 0 & 0 & 0 & 0 \\
				+ & + & 0 & + & 0 & 0 & 0 \\
				0 & 0 & + & 0 & + & 0 & - \\
				0 & 0 & 0 & + & 0 & + & 0 \\
				0 & 0 & 0 & 0 & + & 0 & + \\
				0 & 0 & 0 & - & 0 & + & 0
			\end{bmatrix}
			\]
		\end{minipage}
		\hspace{0.05\textwidth}
		\begin{minipage}{.45\textwidth}
			\centering
		\tikzset{
			vertex/.style={draw, circle, inner sep=0pt, minimum size=0.15cm, fill=black},
			edge/.style={thick}
		}
		\begin{tikzpicture}
			
		\node[vertex, label=above:$u_6$] (v2)  at (-2,1.25) {};
		\node[vertex, label=above:$u_7$] (v1)  at (0,1.25) {};
		\node[vertex, label=below:$u_5$] (v3)  at (-2,0) {};
		\node[vertex, label=below:$u_4$] (v4)  at (0,0) {};
		\node[vertex, label=below:$u_3$] (v5)  at (1.5,0) {};
		\node[vertex, label=below:$u_2$] (v6)  at (3,0) {};
		\node[vertex, label=above:$u_1$] (v7)  at (3,1.25) {};

			\draw[edge] (v1)--node[above]{$ $}(v2);
			\draw[edge] (v2)--node[left]{$ $}(v3);
			\draw[edge] (v3)--node[below]{$ $}(v4);
			\draw[edge] (v4)--node[below]{$ $}(v5);
			\draw[edge] (v1)--node[right]{$ $}(v4);
			
			\draw[edge] (v5)--node[below]{$ $}(v6);
			\draw[edge] (v5)--node[below]{$ $}(v7);
			
		\end{tikzpicture}
		\caption{ $G(\mathcal{P})$.}
		\label{xnfig4.1a.}
		\end{minipage}
	\end{figure}
Then $G(\mathcal{P})$ contains exactly one cycle $C = u_4u_5u_6u_7$ of length $4$, and every composite cycle of length $4$ in the signed directed graph $D(\mathcal{P})$ containing the vertices of $C$ has the same sign. 
Also, the distance from $C$ to every leaf of $G(\mathcal{P})$ is even, and there exists a vertex $u_3 \in V(G(\mathcal{P})) \setminus V(C)$ such that $\dist(u_3,C)$ is odd and $\deg(u_3) \neq 2$. 

Since, every composite cycle of length $6$ in $D(\mathcal{P}) \setminus \{u_1\}$ has the same sign. Therefore, for all $A \in \mathcal{Q}_{\mathrm{SYM}}(\mathcal{P})$, $\det(A(1)) \neq 0$, and hence $m_{A(1)}(0)=0$. 
Thus, $m_{A(1)}(0) < m_A(0)+1$, and therefore $u_1$ is not a $P$-vertex of $A$. Consequently, $\mathcal{P}$ does not require full $P$-vertices.
	
\end{eg}

The following example shows that if the distance from the cycle to any leaf is odd, then $\mathcal{P}$ need not require full $P$-vertices.
\begin{eg}\label{egx1}\rm
	Consider the sign pattern
	$$\mathcal{P}=\begin{bmatrix}
		0 & + & 0 & 0 & 0  \\
		+ & 0 & + & 0 & -  \\
		0 & + & 0 & + & 0  \\
		0 & 0 & + & 0 & +  \\
		0 & - & 0 & + & 0 
	\end{bmatrix}$$
	The underlying signed undirected graph of $\mathcal{P}$ and $\mathcal{P}'$ are given in Figure~\ref{fig1.21}, where $\mathcal{P}'$ is obtained from $\mathcal{P}$ by deleting rows and columns corresponding to $v_1,v_2$.
	\begin{figure}[H]
		\centering
		\tikzset{node distance=1cm}
		
		\begin{tikzpicture}[
			vertex/.style={draw, circle, inner sep=0pt, minimum size=.15cm, fill=black},
			edge/.style={thick}
			]
			\node[vertex, label=above:$v_2$] (v1) at (-3,0) {};
			\node[vertex, label=above:$v_3$] (v2) at (-2,1) {};
			\node[vertex, label=right:$v_4$] (v3) at (-1,0) {};
			\node[vertex, label=below:$v_5$] (v5) at (-2,-1) {};	
			\node[vertex, label=above:$v_1$] (v6) at (-4,0) {};
			
			\draw[edge] (v1)--(v2);
			\draw[edge] (v2)--(v3);
			\draw[edge] (v3)--(v5);
			\draw[edge] (v1)--(v5);
			\draw[edge] (v1)--(v6);
			
			\draw[very thick,->] (0,0) -- (1,0);
			
			vertex/.style={draw, circle, inner sep=0pt, minimum size=.15cm, fill=black},
			edge/.style={thick}
			]
			\node[vertex, label=above:$v_3$] (v2) at (3,0) {};
			\node[vertex, label=above:$v_4$] (v3) at (4,0) {};
			\node[vertex, label=above:$v_5$] (v4) at (5,0) {};
			
			\draw[edge] (v2)--(v3);
			\draw[edge] (v4)--(v3);
		\end{tikzpicture}
		\caption{Underlying signed undirected graph of $\mathcal{P}$, $\mathcal{P}'$, respectively.} \label{fig1.21}
	\end{figure}
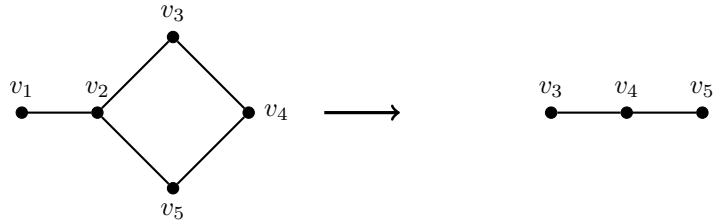
Then $G(\mathcal{P})$ has exactly one cycle $C=v_2v_3v_4v_5$ of length $4$, and every composite cycle of length $4$ in the signed directed graph $D(\mathcal{P})$ containing the vertices of $C$ has the same sign. The distance from $C$ to the leaf $v_1$ is odd. 
If $\mathcal{P}$ requires full $P$-vertices, then Theorem~\ref{th1.1} implies that $\mathcal{P}'$ requires full $P$-vertices. However, since $\mathcal{P}'=P_2$, Corollary~\ref{th0.1e} shows that $\mathcal{P}'$ does not require full $P$-vertices, a contradiction. 
Hence, $\mathcal{P}$ does not require full $P$-vertices.
\end{eg}

The next example shows that if the cycle length is odd, then $\mathcal{P}$ may not require full $P$-vertices.

\begin{eg}\rm \label{eg0.3e}
	Let $\mathcal{P}\in \mathcal{Q}_{5}$ be a symmetric sign pattern with a $0$-diagonal, whose underlying signed undirected graph $G(\mathcal{P})$ is  given in Fig.\ref{xnfig4.1},\\
		\begin{figure}[H]
		\centering
		\begin{minipage}{.45\textwidth}
			\vspace{-1cm}
			\[
			\mathcal{P}=\begin{bmatrix}
				0 & + & 0 & 0 & 0  \\
				+ & 0 & + & 0 & 0  \\
				0 & + & 0 & + & +  \\
				0 & 0 & + & 0 & +  \\
				0 & 0 & + & + & 0 
			\end{bmatrix}
			\]
		\end{minipage}
		\hspace{0.05\textwidth}
		\begin{minipage}{.45\textwidth}
		\centering
		\tikzset{node distance=1cm}
		\begin{tikzpicture}
			\tikzstyle{vertex}=[draw, circle, inner sep=0pt, minimum size=0.15cm, fill=black]
			\tikzstyle{edge}=[thick]
	\node[vertex, label=above:$u_5$](v2) at (-1,1.25) {};
	\node[vertex, label=below:$u_4$](v3) at (-2,0) {};
	\node[vertex, label=below:$u_3$](v4) at (0,0) {};
	\node[vertex, label=below:$u_2$] (v5)  at (1.5,0) {};
	\node[vertex, label=below:$u_1$] (v6)  at (3,0) {};

			\draw[edge] (v2)--node[left]{$ $} (v3);
			\draw[edge] (v2)--node[right]{$ $} (v4);
			\draw[edge] (v3)--node[below]{$ $} (v4);
			
			\draw[edge] (v4)--node[below]{$ $}(v5);
			
			\draw[edge] (v5)--node[below]{$ $}(v6);
			
		\end{tikzpicture}
		\caption{ $G(\mathcal{P})$.}
		\label{xnfig4.1}
		\end{minipage}
	\end{figure}
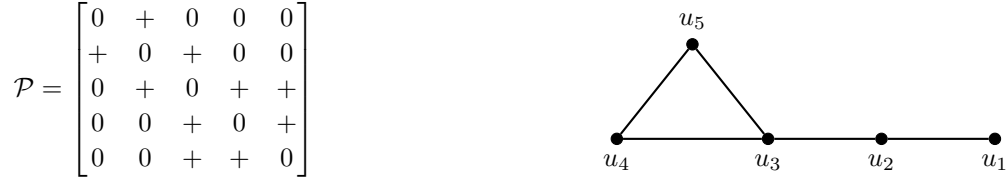
Then $G(\mathcal{P})$ contains exactly one cycle $C = u_3u_4u_5$ of length $3$, and the distance from $C$ to the leaf of $u_1$ is even. 
The signed directed graph $D(\mathcal{P}) \setminus \{u_1\}$ has exactly one composite cycle $(u_2,u_3)(u_4,u_5)$ of length $4$. Therefore, for all $A \in \mathcal{Q}_{\mathrm{SYM}}(\mathcal{P})$, $\det(A(1)) \neq 0$, and hence $m_{A(1)}(0)=0$. 
Thus, $m_{A(1)}(0) < m_A(0)+1$, and therefore $u_1$ is not a $P$-vertex of $A$. Consequently, $\mathcal{P}$ does not require full $P$-vertices.
\end{eg}

Although the condition in Theorem~\ref{th0.6e} does not hold in general when the distance from the cycle to a leaf is odd, as illustrated in Example~\ref{egx1}. We now investigate the conditions under which it remains valid.
We now consider the signed undirected graph obtained by attaching $m$ pendant vertices to a cycle $C_n$.

\begin{theorem}\label{xth3.5}
Let $\mathcal{P}\in \mathcal{Q}_{n+m}$ be a symmetric sign pattern with a $0$-diagonal, whose underlying signed undirected graph is $G(\mathcal{P}) = C_n \cdot mP_1$, where $1 < m \leq n$. 
Then $\mathcal{P}$ requires full $P$-vertices if and only if the distance between any two consecutive hosts is odd.
\end{theorem}

\begin{proof}
    	Let $C_n=v_1v_2\cdots v_n$ and $v_{n+1},v_{n+2},...,v_{n+m}$ be the pendent vertices attached to the vertices $v_{i_1},v_{i_2},...,v_{i_m}$ of $C_n$, respectively (see Figure \ref{fig7.a}).\\
		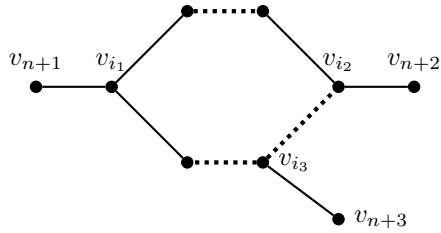
\begin{figure}[H]
		\centering
		\tikzset{node distance=1cm}
		\begin{tikzpicture}[
			vertex/.style={draw, circle, inner sep=0pt, minimum size=.15cm, fill=black},
			edge/.style={thick}
			]
		\node[vertex, label=above:$v_{i_2}$] (v0) at (0,0) {};
		\node[vertex, label=above:$v_{i_1}$] (v1) at (-3,0) {};
			\node[vertex] (v2) at (-2,1) {};
			\node[vertex] (v3) at (-1,1) {};
   	\node[vertex, label=right:$v_{i_3}$] (v4) at (-1,-1){};	 
			\node[vertex] (v5) at (-2,-1) {};	
		\node[vertex, label=above:$v_{n+1}$] (v6) at (-4,0) {};
		\node[vertex, label=above:$v_{n+2}$] (v7) at (1,0) {};
		\node[vertex, label=right:$v_{n+3}$] (v8) at (0,-1.75) {};
			
			\draw[edge] (v3)--(v0);
			\draw[edge, ultra thick, dotted] (v4)--(v0);
			\draw[edge] (v1)--(v2);
			\draw[edge, ultra thick, dotted] (v2)--(v3);
			\draw[edge, ultra thick, dotted] (v4)--(v5);
			\draw[edge] (v5)--(v1);
			\draw[edge] (v1)--(v6);
			\draw[edge] (v0)--(v7);
			\draw[edge] (v4)--(v8);

		\end{tikzpicture}
		\caption{$C_n \cdot mP_1$}
        \label{fig7.a}
	\end{figure}
First, suppose that the distance between any two consecutive hosts is odd. Let $D(\mathcal{P})$ be the underlying signed directed graph of $\mathcal{P}$. Consider
\[
\bar{G} = G(\mathcal{P}) \setminus \{v_{i_1},v_{i_2},\ldots,v_{i_m},v_{n+1},v_{n+2},\ldots,v_{n+m}\}.
\]
Then each connected component of $\bar{G}$ is a path of odd length. Suppose that $\bar{D}$ is the signed directed graph corresponding to $\bar{G}$. Therefore, $\bar{D}$ has a unique composite cycle of length $n-m$, say $\Gamma_1$. 
Let
\[
\Gamma = \Gamma_1 (v_{i_1},v_{n+1}) (v_{i_2},v_{n+2}) \cdots (v_{i_m},v_{n+m}).
\]
Then $\Gamma$ is the unique composite cycle of length $n+m$ in $D(\mathcal{P})$. Hence, for all $A \in \mathcal{Q}_{\mathrm{SYM}}(\mathcal{P})$, we have $m_A(0)=0$.

Since the distance between any two consecutive hosts is odd, $n$ is even (odd, respectively) implies that $m$ is even (odd, respectively). Therefore, $n+m$ is even.
For $1 \leq i \leq n+m$, let $D_i = D(\mathcal{P}) \setminus \{v_i\}$, then $D_i$ has odd order. Moreover, the only simple cycles of odd length in $D_i$ are $(v_1,v_2,\ldots,v_n)$ and $(v_1,v_n,\ldots,v_2)$. Hence, $D_i$ does not contain any composite cycle of length $n+m-1$, which implies that $\det(A(i))=0$.
Since $|m_{A(i)}(0) - m_A(0)| \leq 1$, it follows that $m_{A(i)}(0) = m_A(0)+1$. Therefore, $v_i$ is a $P$-vertex of $A$. Hence, $\mathcal{P}$ requires full $P$-vertices.

For the converse, suppose that $G(\mathcal{P})$ has two consecutive hosts $v_{i_1}$ and $v_{i_2}$ such that $\dist(v_{i_1},v_{i_2})$ is even. Without loss of generality, let
\[
R = \{v_{i_1},v_{i_1+1}\}\{v_{i_1+1},v_{i_1+2}\}\cdots \{v_{i_1+2k+1},v_{i_2}\}
\]
be the path from $v_{i_1}$ to $v_{i_2}$ of even length.

Assume that $\mathcal{P}$ requires full $P$-vertices. Then, by Theorem~\ref{th1.1}, $\mathcal{P}'$ requires full $P$-vertices, where $\mathcal{P}'$ is obtained from $\mathcal{P}$ by deleting the rows and columns corresponding to the vertices $v_{i_1}$ and $v_{n+1}$. Let $G' = G(\mathcal{P}) \setminus \{v_{i_1},v_{n+1}\}$ be the underlying graph of $\mathcal{P}'$.
Since $G'$ is a tree, by Theorem~\ref{thx1}, it has a perfect matching, say $\mathcal{M}'$. Hence,
\[
\{v_{i_1+1},v_{i_1+2}\}, \{v_{i_1+3},v_{i_1+4}\}, \ldots, \{v_{i_1+2k+1},v_{i_2}\} \in \mathcal{M}'.
\]
This implies that $v_{n+2} \notin V(\mathcal{M}')$, which contradicts the fact that $\mathcal{M}'$ is a perfect matching in $G'$. 

Therefore, $\mathcal{P}$ does not require full $P$-vertices.
\end{proof}

We establish an important consequence of Theorem~\ref{xth3.5}. For this purpose, we introduce the following notation.

\begin{defn}\cite[Definition~2.12]{2025}\rm~ 
The corona product $\tilde{C}_n = C_n \odot P_1$ of the cycle $C_n$ by attaching $P_1$ to each vertex of $C_n$.
\end{defn}

\begin{cor}
	Let $\mathcal{P}\in \mathcal{Q}_{2n}$ be a symmetric sign pattern with a $0$-diagonal, whose underlying signed undirected graph is $\tilde{C}_n = C_n \odot P_1$. Then $\mathcal{P}$ requires full $P$-vertices.
\end{cor}

The next theorem establishes a relationship between a symmetric sign pattern $\mathcal{P}\in\mathcal{Q}_n$ that requires full $P$-vertices and a principal submatrix of $\mathcal{P}$ of order $n-2$.

\begin{theorem}\label{th0.5e}
	Let $\mathcal{P}=[p_{ij}]$ be an $n\times n$ symmetric sign pattern whose underlying signed undirected graph is $G(\mathcal{P})$. Suppose that $G(\mathcal{P})$ has an edge $\{u,v\}$ with $p_{uv}=p_{vu}=+$ and $deg(u)=deg(v)=2$ in $G(\mathcal{P})$. Assume that $u$ is adjacent to $u_1$, where $u_1\neq v$, and $v$ is adjacent to $v_1$, where $v_1\neq u,u_1$.
	
	If $p_{u,u_1}=p_{v,v_1}$ and $p_{u_1,v_1}\in \{0,-\}$, then $\mathcal{P}$ requires full $P$-vertices if and only if $\mathcal{P}'$ requires full $P$-vertices, where $\mathcal{P'}$ is obtained from $\mathcal{P}$ by deleting the rows and columns corresponding to $u,v$ and setting $p_{u_1,v_1}=p_{v_1,u_1}=-$.
	
	If $p_{u,u_1}\neq p_{v,v_1}$ and $p_{u_1,v_1}\in \{0,+\}$, then $\mathcal{P}$ requires full $P$-vertices if and only if $\mathcal{P}'$ requires full $P$-vertices, where $\mathcal{P'}$ is obtained from $\mathcal{P}$ by deleting the rows and columns corresponding to $u,v$ and setting $p_{u_1,v_1}=p_{v_1,u_1}=+$.
\end{theorem}
\begin{proof}
	First, suppose that $\mathcal{P}$ requires full $P$-vertices. We prove the case in which $p_{u,u_1}=p_{v,v_1}$.  Let $B=[b_{ij}]\in \mathcal{Q}_{\mathrm{SYM}}(\mathcal{P}')$.
		Without loss of generality, assume that $u=1$, $v=2$, $u_1=3$, and $v_1=4$.
	Therefore, $b_{12}<0$.
	
	\textbf{Case~1:} $p_{u_1v_1}=0$.
	\\
	Consider 
	$$A=\left[
	\begin{array}{cc|ccccc}
		0 & -\frac{ab}{b_{12}} & a & 0 &0 & \cdots & 0\\
		-\frac{ab}{b_{12}} & 0 & 0 & b &0 & \cdots & 0\\ \hline
		a & 0 & b_{11}  & 0  &b_{13}&\cdots        &  b_{1,n-2} \\
		0 & b & 0  & b_{22}  &  b_{23}&\cdots      &b_{2,n-2}   \\
        0 & 0 & b_{13}& b_{23}& b_{33}&\cdots     &b_{3,n-2}\\
		\vdots & \vdots &\vdots & \vdots&\vdots & \ddots&\vdots \\
		0 & 0 &  b_{1,n-2} &b_{2,n-2} &b_{3,n-2} & \dots       & b_{n-2,n-2}
	\end{array}\right]=\left[
	\begin{array}{c |c }
		A_{11}&C^{\mathsf T}\\
		
		\hline
		C&B'\\
	\end{array} \right]~(say),$$ where $a,b$ are nonzero real numbers with $\sgn(a)=p_{13}$ and $\sgn(b)=p_{24}$. Then $A\in \mathcal{Q}_{\mathrm{SYM}}(\mathcal{P})$, also  
	\begin{equation}\label{eq1.3}
		\left[
		\begin{array}{c |c }
			A_{11}&\textbf{0}^{\mathsf T}\\
			
			\hline
			\textbf{0}&B'-CA_{11}^{-1}C^{\mathsf T}\\

		\end{array} \right]=\left[
		\begin{array}{c |c }
			I_2&\textbf{0}^{\mathsf T}\\
			
			\hline
			-CA_{11}^{-1}&I_{n-2}\\

		\end{array} \right]\left[
		\begin{array}{c |c }
			A_{11}&C^{\mathsf T}\\
			
			\hline
			C&B'\\

		\end{array} \right]\left[
		\begin{array}{c |c }
			I_2&-(CA_{11}^{-1})^{\mathsf T}\\
			
			\hline
			\textbf{0}&I_{n-2}\\

		\end{array} \right],\end{equation}
	where
	$
	A_{11}^{-1}
	=
	\begin{bmatrix}
		0 & -\frac{b_{12}}{ab} \\
		-\frac{b_{12}}{ab} & 0
	\end{bmatrix}$ and
	$B' - C A_{11}^{-1} C^{\mathsf T}=B$. Therefore, $A$ is congruent to 
    \[
	\begin{bmatrix}
		0 & -\frac{ab}{b_{12}} \\
		-\frac{ab}{b_{12}} & 0
	\end{bmatrix}
	\oplus B.
	\]
	
	For $3\le i\le n$, by an argument similar to that in~\eqref{eq1.3},
	we can show that $A(i)$ is congruent to
	\[
	\begin{bmatrix}
		0 & -\frac{ab}{b_{12}} \\
		-\frac{ab}{b_{12}} & 0
	\end{bmatrix}
	\oplus B(i-2).
	\]
	Since, $v_i$ is a $P$-vertex of $A$, we have $m_{A(i)}(0)=m_{A}(0)+1$. Also, $A_{11}$ is nonsingular, hence $m_{B(i-2)}(0)=m_B(0)+1$.
	Therefore, $v_i$ is a $P$-vertex of $B$, for each $i=3,4,\ldots,n$.
	Equivalently,
	$
	P_{\nu}(B)=n-2,
	$
	hence $\mathcal{P}'$ requires full $P$-vertices.
	
	\textbf{Case~2:} $p_{u_1v_1}=-$.
	\\
	Consider 
	$$A=\left[
	\begin{array}{cc|ccccc}
		0 & -\frac{2ab}{b_{12}} & a & 0 &0 & \cdots & 0\\
		-\frac{2ab}{b_{12}} & 0 & 0 & b &0 & \cdots & 0\\ \hline
		a & 0 & b_{11}  & \frac{b_{12}}{2}  &b_{13}&\cdots        &  b_{1,n-2} \\
		0 & b & \frac{b_{12}}{2}  & b_{22}  &  b_{23}&\cdots      &b_{2,n-2}   \\
		\vdots & \vdots &\vdots & \vdots&\vdots & \ddots&\vdots \\
		0 & 0 &  b_{1,n-2} &b_{2,n-2} &b_{2,n-2} & \dots       & b_{n-2,n-2}
	\end{array}\right],$$
	where $a,b$ are nonzero real numbers with $\sgn(a)=p_{13}$ and $\sgn(b)=p_{24}$.
	Then $A \in \mathcal{Q}_{\mathrm{SYM}}(\mathcal{P})$ and, by a congruence
	argument similar to that in~\eqref{eq1.3}, $A$ is congruent to
	\[
	\begin{bmatrix}
		0 & -\frac{2ab}{b_{12}}\\
		-\frac{2ab}{b_{12}} & 0
	\end{bmatrix}
	\oplus B.
	\] Then similarly to Case~1, we can show that $\mathcal{P}'$ requires full $P$-vertices.
	
	Conversely, suppose that $\mathcal{P}'$ requires full $P$-vertices. Let $A\in \mathcal{Q}_{\mathrm{SYM}}(\mathcal{P})$, then the matrix $A$ can be written as
	$$
	A=\left[
	\begin{array}{c c|c c c c}
		0&a_{12}&a_{13}&0&\cdots&0\\
		a_{12}&0&0&a_{24}&\cdots&0\\
		\hline
		a_{13}&0&&&&\\
		
		0&a_{24}&&&\\
		
		\vdots&\vdots&&B&\\
		
		0&0&&&
	\end{array} \right]=\left[
	\begin{array}{c |c }
		A_{11}&C_1^{\mathsf T}\\
		
		\hline
		C_1&B\\
	\end{array} \right]~(say),$$ where $B=[b_{ij}]$ is a symmetric matrix of order $n-2$. Then  
	\begin{equation}\label{eq1.31}
		\left[
		\begin{array}{c |c }
			A_{11}&\textbf{0}^{\mathsf T}\\
			
			\hline
			\textbf{0}&\hat{B}\\

		\end{array} \right]=\left[
		\begin{array}{c |c }
			I_2&\textbf{0}^{\mathsf T}\\
			
			\hline
			-C_1A_{11}^{-1}&I_{n-2}\\

		\end{array} \right]\left[
		\begin{array}{c |c }
			A_{11}&C_1^{\mathsf T}\\
			
			\hline
			C_1&B\\

		\end{array} \right]\left[
		\begin{array}{c |c }
			I_2&-(C_1A_{11}^{-1})^{\mathsf T}\\
			
			\hline
			\textbf{0}&I_{n-2}\\

		\end{array} \right],\end{equation}
	where
	$$
	\hat{B} = B - C_1 A_{11}^{-1} C_1^{\mathsf T}=\begin{bmatrix}
		b_{11}&b_{12}-\frac{a_{13}a_{24}}{a_{12}}& b_{13}& \cdots&b_{1,n-2}\\
		b_{12}-\frac{a_{13}a_{24}}{a_{12}}&b_{22}&b_{23}&\cdots&b_{2,n-2}\\
		b_{13}&b_{23}&b_{33}&\cdots&b_{3,n-2}\\
		\vdots&\vdots&\vdots&\ddots&\vdots\\
		b_{1,n-2}&b_{2,n-2}&b_{3,n-2}&\cdots&b_{n-2,n-2}
		
	\end{bmatrix}.
	$$ Therefore, $A$ is congruent to $A_{11}\oplus \hat{B}$.
	Since $-\frac{a_{13}a_{24}}{a_{12}}<0$ and
	$p_{u_1 v_1}\in\{0,-\}$, it follows that
	$
	\hat{B}\in\mathcal{Q}_{\mathrm{SYM}}(\mathcal{P}').
	$

	For $i=1$,
	let $S_{23}$ be the permutation matrix obtained from
	$I_{n-1}$ by alternating rows two and three.
	Then $$S_{23}A(1)S_{23}^{\mathsf{T}}=\left[
	\begin{array}{cc|cccc}
		0 & a_{24} & 0 & 0 & \cdots & 0\\
		a_{24} & b_{22} & b_{12} & b_{23} & \cdots & b_{2,n-2}\\ \hline
		0 & b_{12} &   &   &        &   \\
		0 & b_{23} &   &   &        &   \\
		\vdots & \vdots & & B(2) & & \\
		0 & b_{2,n-2} &   &   &        &
	\end{array}\right]
	= \left[
	\begin{array}{cc|c}
		0 & a_{24} & \mathbf{0}^{\mathsf T}\\
		a_{24} & b_{22} & C_{2}^{\mathsf T}\\ \hline
		\mathbf{0} & C_{2} & B(2)
	\end{array}\right],$$        
	Arguing as in \eqref{eq1.22}, the matrix $S_{23}A(1)S_{23}^\mathsf{T}$ is congruent to \begin{equation}\label{equ1}
		\begin{bmatrix}
		0&a_{24}\\a_{24}&b_{22}
	\end{bmatrix}\oplus B(2).
\end{equation}  
	Since $B(2)=\hat{B}(2)$ and $v_1$ is a $P$-vertex of $\hat{B}$, it follows that $
	m_{\hat{B}(2)}(0) = m_{\hat{B}}(0) + 1$.
	Hence, from \eqref{equ1},
	$m_{A(1)}(0) = m_A(0) + 1.$  Therefore, $u$ is a $P$-vertex of $A$.   
	
	For $i=2$, 
	$$A(2)=\left[
	\begin{array}{cc|cccc}
		0 & a_{13} & 0 & 0 & \cdots & 0\\
		a_{13} & b_{11} & b_{12} & b_{13} & \cdots & b_{1,n-2}\\ \hline
		0 & b_{12} &   &   &        &   \\
		0 & b_{13} &   &   &        &   \\
		\vdots & \vdots & & B(1) & & \\
		0 & b_{1,n-2} &   &   &        &
	\end{array}\right]
	= \left[
	\begin{array}{cc|c}
		0 & a_{13} & \mathbf{0}^{\mathsf T}\\
		a_{13} & b_{11} & C_{3}^{\mathsf T}\\ \hline
		\mathbf{0} & C_{3} & B(1)
	\end{array}\right],$$
	similarly, $A(2)$ is congruent to $$\begin{bmatrix}
		0&a_{13}\\a_{13}&b_{11}
	\end{bmatrix}\oplus B(1).$$   
	Since $B(1)=\hat{B}(1)$ and $u_1$ is a $P$-vertex of $\hat{B}$, it follows that $
	m_{\hat{B}(1)}(0) = m_{\hat{B}}(0) + 1$.
	Hence,
	$m_{A(2)}(0) = m_A(0) + 1.$  Therefore, $v$ is a $P$-vertex of $A$.

	For every $3 \le i \le n$, by an argument similar to~\eqref{eq1.31}, the matrix
	$A(i)$ is congruent to
	$
	\begin{bmatrix}
		0 & a_{12}\\
		a_{12} & 0
	\end{bmatrix}
	\oplus \hat{B}(i-2)
	$.
	Since $v_i$ is a $P$-vertex of $B'$, hence $m_{\hat{B}(i-2)}(0) = m_{\hat{B}}(0) + 1$, and consequently
	$
	m_{A(i)}(0) = m_A(0) + 1$. Thus, $v_i$ is a $P$-vertex of $A$, for $3 \le i \le n$. Therefore, $\mathcal{P}$ requires full $P$-vertices.	
\end{proof}
The following corollary is an immediate consequence of Theorem~\ref{th0.5e}.
\begin{cor}\label{th0.51e}
	Let $\mathcal{P}=[p_{ij}]$ be an $n\times n$ symmetric sign pattern whose underlying signed undirected graph is $G(\mathcal{P})$. Suppose that $G(\mathcal{P})$ has an edge $\{u,v\}$ with $p_{uv}=p_{vu}=-$ and $deg(u)=deg(v)=2$ in $G(\mathcal{P})$. Assume that $u$ is adjacent to $u_1$, where $u_1\neq v$, and $v$ is adjacent to $v_1$, where $v_1\neq u,u_1$.
	
	If $p_{u,u_1}=p_{v,v_1}$ and $p_{u_1,v_1}\in \{0,+\}$, then $\mathcal{P}$ requires full $P$-vertices if and only if $\mathcal{P}'$ requires full $P$-vertices, where $\mathcal{P'}$ is obtained from $\mathcal{P}$ by deleting the rows and columns corresponding to $u,v$ and setting $p_{u_1,v_1}=p_{v_1,u_1}=+$.
	
	If $p_{u,u_1}\neq p_{v,v_1}$ and $p_{u_1,v_1}\in \{0,-\}$, then $\mathcal{P}$ requires full $P$-vertices if and only if $\mathcal{P}'$ requires full $P$-vertices, where $\mathcal{P'}$ is obtained from $\mathcal{P}$ by deleting the rows and columns corresponding to $u,v$ and setting $p_{u_1,v_1}=p_{v_1,u_1}=-$.
\end{cor}

\begin{eg}\rm
	Suppose that $\mathcal{P}\in \mathcal{Q}_6$ is a symmetric sign pattern with a $0$-diagonal. From left to right, let $\mathcal{P}$, $\mathcal{P}_1$, and $\mathcal{P}_2$ denote the sign patterns corresponding to the signed directed graphs $D$, $D_1$, and $D_2$, respectively, shown in Fig.~\ref{fig8}. 
	Let $G$, $G_1$, and $G_2$ be the underlying signed undirected graphs corresponding to $\mathcal{P}$, $\mathcal{P}_1$, and $\mathcal{P}_2$, respectively.
	
	\begin{figure}[H] 
		\tikzset{node distance=1cm}
		\centering
		\usetikzlibrary{decorations.markings}
		
		\tikzset{
			vertex/.style={
				draw, circle, fill=black,
				inner sep=0pt, minimum size=0.15cm
			},
			edge/.style={
				thick,
				postaction={
					decorate,
					decoration={
						markings,
						mark=at position 0.75 with {\arrow{stealth}}
					}
				}
			}
		}
		
		\begin{tikzpicture}[scale=1, every node/.style={font=\small}]
			\node[vertex, label=right:$u$](v1)at(-3,0){};
			\node[vertex, label=above:$u_1$](v2)at(-4,1){};
			\node[vertex](v3)at(-5.5,1){};
			\node[vertex](v4)at(-6.5,0){};
			\node[vertex, label=below:$v_1$](v5)at(-5.5,-1){};	\node[vertex, label=below:$v$](v6)at(-4,-1){};
			
			\draw[edge, bend left=15] (v1) to node[left] {$+$} (v2);
			\draw[edge, bend left=15] (v2) to node[right] {$+$} (v1);
			
			\draw[edge, bend left=15] (v2) to node[below] {$+$} (v3);
			\draw[edge, bend left=15] (v3) to node[above] {$+$} (v2);
			
			\draw[edge, bend left=15] (v3) to node[right] {$+$} (v4);
			\draw[edge, bend left=15]  (v4) to node[left] {$+$} (v3);
			
			\draw[edge, bend left=15] (v4) to node[right] {$+$} (v5);
			\draw[edge, bend left=15] (v5) to node[left] {$+$} (v4);
			
			\draw[edge, bend left=15] (v5) to node[above] {$+$} (v6);
			\draw[edge, bend left=15] (v6) to node[below] {$+$} (v5);
			
			\draw[edge, bend left=15] (v6) to node[left] {$+$} (v1);
			\draw[edge, bend left=15] (v1) to node[right] {$+$} (v6);

			\draw[very thick,->] (-2,0) -- (-1,0) ;

			\node[vertex](v1)at(0,0){};
			\node[vertex](v2)at(1,1){};
			\node[vertex, label=right:$u_1$](v3)at(2.5,1){};
			\node[vertex, label=below:$v_1$](v4)at(1,-1){};	
			
			\draw[edge, bend left=15] (v1) to node[left] {$+$} (v2);
			\draw[edge, bend left=15] (v2) to node[right] {$+$} (v1);
			
			\draw[edge, bend left=15] (v2) to node[above] {$+$} (v3);
			\draw[edge, bend left=15] (v3) to node[below] {$+~~~$} (v2);
			
			\draw[edge, bend left=10] (v3) to node[right] {$-$} (v4);
			\draw[edge, bend left=10]  (v4) to node[left] {$-$} (v3);
			
			\draw[edge, bend left=15] (v1) to node[right] {$+$} (v4);
			\draw[edge, bend left=15] (v4) to node[left] {$+$} (v1);

			\draw[very thick,->] (4,0) -- (5,0) ;

			\node[vertex](v1)at(6,0){};
			\node[vertex](v2)at(7.5,0){};
			
			\draw[edge, bend left=15] (v1) to node[above] {$+$} (v2);
			\draw[edge, bend left=15] (v2) to node[below] {$+$} (v1);
		\end{tikzpicture}
		\caption{Determining whether $\mathcal{P}$ requires full $P$-vertices using Theorem~\ref{th0.5e} and Corollary~\ref{th0.51e}.}
		\label{fig8}
	\end{figure}
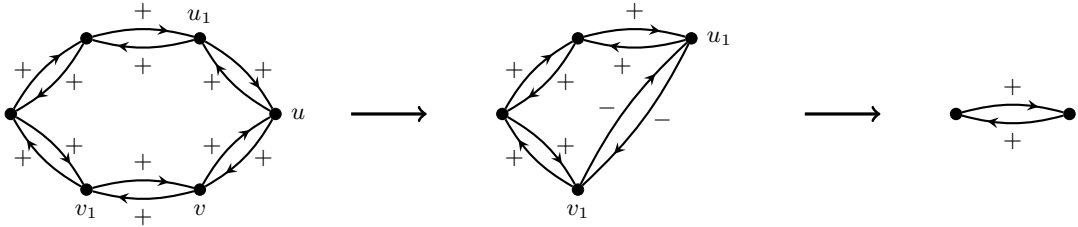
	Applying Theorem~\ref{th0.5e} to the edge $\{u,v\}$ of $G$, we obtain that $\mathcal{P}$  requires full $P$-vertices if and only if $\mathcal{P}_1$  requires full $P$-vertices. 
	Again, applying Corollary~\ref{th0.51e} to the edge $\{u_1,v_1\}$ of $G_1$, it follows that $\mathcal{P}_1$  requires full $P$-vertices if and only if $\mathcal{P}_2$  requires full $P$-vertices. 
	Since $G_2$ is a path of even length, Theorem~\ref{th0.1e} implies that $\mathcal{P}_2$  requires full $P$-vertices. Consequently, $\mathcal{P}$  requires full $P$-vertices.
\end{eg}

We now establish an interesting consequence of Theorem~\ref{th0.5e}. For this purpose, we introduce the following notations.

\begin{defn}
	Suppose that $\mathcal{P}$ is a symmetric sign pattern, whose underlying signed undirected graph is $G(\mathcal{P})$. Then any two cycles $\mathcal{C}_1$, $\mathcal{C}_2$ in $G(\mathcal{P})$ are called \textit{path-adjacent} if there exists a path $$\{u,v_1\}\{v_1,v_2\}\cdots \{v_{r-1},v_{r}\}\{v_{r},w\}$$ in $G(\mathcal{P})$ where $u\in V(\mathcal{C}_1)$, $w\in V(\mathcal{C}_2)$ and $v_1,v_2,...,v_{r}$ are not vertices of any cycle in $G(\mathcal{P})$.
\end{defn}

\begin{defn}
    	A \textit{dumbbell graph}, denoted by $D_{l,m,n}$, is a bicyclic graph obtained by joining two vertex-disjoint cycles $C_l$ and $C_m$ by a path $P_{n+3}$ of length $n+3$, having only its end-vertices in common with the two cycles, where $l,m \geq 3$ and $n \geq -1$ (see Figure \ref{fig12.a}).
\end{defn}
	\begin{figure}[H]
	\centering
	\begin{tikzpicture}
		\tikzstyle{vertex}=[draw, circle, inner sep=0pt, minimum size=.15cm, fill=black]
			
			\node[vertex,label=left:$u_1$](v1)at(0,0){};
		\node[vertex,label=above:$u_2$](v2)at(-0.75,1){};
		\node[vertex,label=above:$u_3$](v3)at(-2,1){};
		\node[vertex](v4)at(-2,-1){};
		\node[vertex,label=left:$u_4$](v5)at(-2.75,0){};
		\node[vertex,label=below:$u_l$](v6)at(-0.75,-1){};

		\draw[thick](v1)--(v2)--(v3)--(v5);
		\draw[thick,dotted](v5)--(v4)--(v6);
		\draw[thick](v6)--(v1);
			
			\node[vertex,label=right:$w_1$](u1)at(4.5,0){};
		\node[vertex,label=above:$w_2$](u2)at(5.25,1){};
		\node[vertex,label=above:$w_3$](u3)at(6.5,1){};
		\node[vertex](u4)at(6.5,-1){};
		\node[vertex,label=right:$w_4$](u5)at(7.25,0){};
		\node[vertex,label=below:$w_m$](u6)at(5.25,-1){};

		\draw[thick](u1)--(u2)--(u3)--(u5);
		\draw[thick,dotted](u5)--(u4)--(u6);
		\draw[thick](u6)--(u1);
			
			\node[vertex,label=above:$v_1$](p1) at (1.5,0) {};
			\node[vertex,label=above:$v_{n+1}$](p2) at (3,0) {};
			
			\draw[thick](v1)--(p1);
			\draw[thick](p2)--(u1);
			\draw[thick,dotted](p1)--(p2);
		\end{tikzpicture}
        \caption{$D_{l,m,n}$}
        \label{fig12.a}
\end{figure}
	
\begin{theorem}\label{th0.61e}
		Let $\mathcal{P}\in\mathcal{Q}_n$ be a nonnegative symmetric sign pattern with a $0$-diagonal, whose underlying signed undirected graph is $G(\mathcal{P})= D_{4p+2,4q+2,2l-1}$, where $p,q,l\in \mathbb{N}$. Then $\mathcal{P}$ requires full $P$-vertices.
\end{theorem}

\begin{proof}
    Suppose that the cycle $C_{4p+2}=u_1u_2\cdots u_{4p+2}$ is path-adjacent to the cycle $C_{4q+2}$. Let $u_1 R w_1$ be the path joining $C_{4p+2}$ and $C_{4q+2}$, where
	\[
	R = \{v_1,v_2\}\{v_2,v_3\} \cdots \{v_{2l-1},v_{2l}\},
	\]
	such that $u_1 \in V(C_{4p+2})$, $w_1 \in V(C_{4q+2})$, and $v_i \notin V(C_{4p+2}) \cup V(C_{4q+2})$ for all $i = 1,2,\ldots,2l$.
    
    Since $p_{u_2u_3}=+$ and $\deg(u_2)=\deg(u_3)=2$ in $G(\mathcal{P})$, where $u_2$ is adjacent to $u_1$, and $u_3$ is adjacent to $u_4$, $p_{u_1u_2}=p_{u_3u_4}=+$, and $p_{u_1u_4}=0$, it follows from Theorem~\ref{th0.5e} that $\mathcal{P}$ requires full $P$-vertices if and only if $\mathcal{P}^1$ requires full $P$-vertices, where $\mathcal{P}^1=[p_{ij}^1]\in \mathcal{Q}_{n-2}$ is the principal subpattern of $\mathcal{P}$ obtained by deleting the rows and columns corresponding to $u_2$ and $u_3$, and setting $p_{u_1u_4}=-$.
	Therefore, the underlying graph of $\mathcal{P}^1$, denoted by $G^1$, is obtained from $G(\mathcal{P})$ by replacing $C_{4p+2}$ with the cycle $C^1:u_1u_4u_5\cdots u_{4p+2},$ which has length $4p$.
	
	In $\mathcal{P}^1$, since $p_{u_4u_5}^1 = +$ and $\deg(u_4)=\deg(u_5)=2$ in $G^1$, where $u_4$ is adjacent to $u_1$, and $u_5$ is adjacent to $u_6$, $p_{u_1u_4}^1 \neq p_{u_5u_6}^1$, and $p_{u_1u_6}^1 \in \{0,+\}$, it follows from Theorem~\ref{th0.5e} that $\mathcal{P}^1$ requires full $P$-vertices if and only if $\mathcal{P}^2$ requires full $P$-vertices, where $\mathcal{P}^2 = [p_{ij}^2] \in \mathcal{Q}_{n-4}$ is the principal subpattern of $\mathcal{P}^1$ obtained by deleting the rows and columns corresponding to $u_4$ and $u_5$, and setting $p_{u_1u_6}^1 = +$.
	Therefore, the underlying graph of $\mathcal{P}^2$, denoted by $G^2$, is obtained from $G^1$ by replacing the cycle $C^1$ with the cycle $C^2 : u_1u_6u_7 \cdots u_{4p+2}$,
	in which all edges are positive.
	
	Since the cycle $C$ has $4p+2$ vertices, by continuing this process iteratively, after $2p$ steps we obtain that $\mathcal{P}$ requires full $P$-vertices if and only if $\mathcal{P}^{2p}$ requires full $P$-vertices, where $\mathcal{P}^{2p} = [p_{ij}^{2p}] \in \mathcal{Q}_{n-4p}$ is the principal subpattern of $\mathcal{P}$ obtained by deleting the rows and columns corresponding to $u_2, u_3, \ldots, u_{4p+1}$, and setting $p_{u_1u_{4p+2}} = +$.
	Therefore, the underlying graph of $\mathcal{P}^{2p}$, denoted by $G^{2p}$, is obtained from $G(\mathcal{P})$ by replacing the cycle $C_{4p+2}$ with the edge $\{u_1,u_{4p+2}\}$.

    Therefore, $G^{2p}=C_{4q+2}\cdot P_{2l+2}$, by Corollary~\ref{cor1x}, $\mathcal{P}^{2p}$ requires full $P$-vertices, hence $\mathcal{P}$ requires full $P$-vertices.
\end{proof}

\section{Conclusions}\label{s5}

We have studied the full $P$-vertex problem for symmetric sign patterns and established structural criteria for requiring full $P$-vertices. In particular, we obtained necessary and sufficient conditions for important classes of graphs, including trees and unicyclic graphs. We showed that a tree sign pattern with a zero diagonal requires full $P$-vertices if and only if its underlying graph admits a perfect matching. We also derived several conditions for sign patterns whose underlying graph contain cycles and characterized such pattern requiring full $P$-vertices.

Although Theorem~\ref{th0.61e} addresses sign patterns whose underlying graph contain more than one cycle, the general problem remains open. Further research may focus on broader classes of graphs, including bicyclic and more general cyclic structures, to determine conditions under which they require full $P$-vertices.

\newpage

\end{document}